\documentclass[reqno,11pt]{amsart}

\usepackage{mathdots}
\usepackage{tikz}
\usepackage{tikz-cd}
\usepackage{amssymb}
\usepackage{amsgen}
\usepackage{amsmath}
\usepackage{amsthm}
\usepackage{mathrsfs}
\usepackage{cite}
\usepackage{amsfonts}
\usepackage{enumitem}

\newcommand{\coker}{\mathop{\mathrm{coker}}\nolimits}
\newcommand{\cd}{\mathop{\mathrm{cd}}}

\newcommand{\End}{\mathop{\mathrm{End}}}

\newcommand{\rad}{\mathop{\mathrm{rad}}\nolimits}

\newcommand{\Res}{\mathop{\mathrm{Res}}\nolimits}
\newcommand{\Ind}{\mathop{\mathrm{Ind}}\nolimits}
\newcommand{\Coind}{\mathop{\mathrm{Coind}}\nolimits}

\newcommand{\J}{\mathrel{\mathscr J}} 
\newcommand{\eL}{\mathrel{\mathscr L}} 

\newcommand{\inv}{^{-1}}
\newcommand{\p}{\varphi}

\newcommand{\ov}[1]{\ensuremath{\overline {#1}}}
\newcommand{\til}[1]{\ensuremath{\widetilde {#1}}}

\newcommand{\wh}{\widehat}

\newcommand{\soc}[1]{\mathrm{soc}(#1)}

\newcommand{\Hom}{\mathop{\mathrm{Hom}}\nolimits}
\newcommand{\Tor}{\mathop{\mathrm{Tor}}\nolimits}
\newcommand{\Ext}{\mathop{\mathrm{Ext}}\nolimits}

\usepackage{xcolor}

\newtheorem{Thm}{Theorem}[section]
\newtheorem{Prop}[Thm]{Proposition}

\newtheorem{Lemma}[Thm]{Lemma}
{\theoremstyle{definition}
}
{\theoremstyle{remark}
}
\newtheorem{Cor}[Thm]{Corollary}
{\theoremstyle{remark}
}
{\theoremstyle{remark}
\newtheorem{Example}[Thm]{Example}}

\newtheorem{Conjecture}[Thm]{Conjecture}

{\theoremstyle{remark}
}
{\theoremstyle{remark}
}
\newtheorem{Question}[Thm]{Question}
{\theoremstyle{remark}
\newtheorem*{Claim*}{Claim}}

\numberwithin{equation}{section}

\title{The modular representation theory of monoids}

\author{Benjamin Steinberg}
\address[B.~Steinberg]{%
    Department of Mathematics\\
    City College of New York\\
    Convent Avenue at 138th Street\\
    New York, New York 10031\\
    USA}
\email{bsteinberg@ccny.cuny.edu}

\thanks{The author was supported by a Simons Foundation Collaboration Grant, award number 849561.}
\date{July 10, 2023}


\begin{document}

\begin{abstract}
This paper develops the fundamentals of modular representation theory for finite monoids, introducing the decomposition matrix and exploring its connection to Brauer characters. We define modular characteristic and explain how the representation theory in nonmodular positive characteristic behaves like the characteristic zero theory by showing that one can lift from nonmodular characteristic $p$ all simple and projective indecomposable modules, as well as a quiver presentation of the basic algebra.  As an application of the theory developed, we give a new proof of Glover's theorem that the monoid of $2\times 2$-matrices over $\mathbb F_p$ has infinite representation type over fields of characteristic $p$. 

We also investigate the relationship between nonsingularity of the Cartan matrix of a monoid algebra in characterstic zero and in positive characteristic.  We show that for von Neumann regular monoids the Cartan matrix is always nonsingular and we show that if a monoid has aperiodic left (or right) stabilizers, then nonsingularity in characteristic zero implies nonsingularity in positive characteristic.  Florian Eisele has recently shown that the Cartan matrix of a monoid algebra can be nonsingular in characteristic zero and singular in positive characteristic, disproving a conjecture of the author in an earlier version of this paper.  A new conjecture is proposed, unifying the cases of regular monoids and monoids with aperiodic stabilizers.
\end{abstract}
\maketitle

\section{Introduction}
Outside of some work on characters~\cite{guralnick,ourcharacter} and Brauer characters~\cite{guralnick,PutchaBrauer}, there has not been much research into the modular representation theory of finite monoids.  There are most likely two principal reasons for this.  First and foremost, the applications of monoid representation theory have been to areas like automata theory~\cite{mortality,mycerny,AMSV}, combinatorics and probability theory~\cite{Brown1,ayyer_schilling_steinberg_thiery.2013,RhodesSchilling,ourmemoirs} and symbolic dynamics~\cite{BRzeta}; see~\cite[Chapter~14]{repbook} for details.  Almost all these applications involve the characteristic zero theory, with the exception of one result from~\cite{AMSV} that used modular representations to count modulo a prime the number of factorizations of a word into a product of elements from some formal languages.
The second reason is that one of the main strategies in the representation theory of finite monoids, since its inception, has been to reduce problems to the representation theory of groups.  In characteristic zero, this has been quite successful because the ordinary representation theory of groups is a well-understood topic for which the vast majority of questions can be answered.  Modular representation theory of groups is a much more subtle subject and there is still a large amount of effort being dedicated toward improving our understanding.  Therefore, reducing questions to the modular representation theory of groups is not guaranteed to result in an answer.

A key ingredient that distinguishes modular representation theory of groups from the study of arbitrary self-injective finite dimensional algebras is the ability to pass between the characteristic zero theory and the characteristic $p$ theory.  Since the characteristic zero theory is very powerful, it can be a huge boon in solving problems about modular representations.  This began with the work of Brauer and his development of Brauer characters and continues to pervade much of the work in this area.

While there are no conceptual barriers to trying this approach for monoids, there is the obvious difficulty that monoid algebras are seldom semisimple in characteristic zero, and so at first sight it may seem that you are going from a complicated algebra in characteristic $p$ to an equally complicated algebra in characteristic zero.  But this is not really the case.  For example, a von Neumann regular monoid has a quasi-hereditary algebra in characteristic zero, but not in modular characteristic, and so lifting to characteristic zero still simplifies things.  Similarly, the ability to reduce certain types of questions to groups means that lifting to characteristic zero may allow you to apply the ordinary representation theory of groups.

One of the goals of this paper is to put forward the idea that good behavior in characteristic zero will, to some extent, be reflected in positive characteristic.  In an earlier version of this paper,  I proposed the conjecture that if $M$ is a finite monoid and the Cartan matrix of the complex algebra $\mathbb CM$ is nonsingular, then the Cartan matrix of $KM$ will be nonsingular for every field $K$.  This of course happens for groups thanks to work of Brauer. In this paper I prove that the Cartan matrix is always nonsingular for algebras of von Neumann regular monoids, which include, for example, matrix monoids over finite fields, finite monoids of Lie type~\cite{LieType}, full transformation monoids and most diagram monoids.  I also prove that nonsingularity of the Cartan matrix in characteristic zero implies nonsingularity in positive charactersitic for monoids in which the left (or right) stabilizer of every element is aperiodic.  This means if $mn=n$ for some $m,n\in M$, then $m^a=m^{a+1}$ for some $a\geq 0$ (for example, this is the case for groups).  The former result relies on connections between the Cartan matrix of $eAe$ and $A/AeA$ when $AeA$ is projective as a left $A$-module~\cite{quasistrat}, whereas the latter uses the theory of Brauer characters and lifting projective modules to characteristic zero and is based on Brauer's proof of the corresponding result for groups.  Recently, Eisele~\cite{eisele2023counterexample} provided an example of a monoid whose Cartan matrix is nonsingular in characteristic zero but is singular in a certain positive characteristic, disproving my original conjecture.  His construction is built on taking a group $G$ and a sufficiently complicated $G$-biset $X$ and turning $G\cup X\cup \{0\}$ into a monoid with $X^2=0$.  The proof exploits the decomposition matrix techniques developed herein.  I propose a new conjecture that covers the positive results for regular monoids and monoids with aperiodic stabilizers and which avoids the pitfalls of Eisele's construction.

The main thrust of this paper is to develop the fundamentals of modular representation theory for monoids.  Some of this work is straightforward modification of the group case to orders in a not necessarily semisimple finite dimensional algebra over the field of fractions of a complete discrete valuation ring.  This involves defining the decomposition map and constructing a commutative square generalizing the Brauer cde triangle (one only obtains a triangle for orders in a semisimple algebra).  A suggestion on how to do this can be found  in~\cite{WebsterMO}, and I think this should be considered folklore.   For monoid algebras, the decomposition map is surjective, essentially because Brauer proved it is so for groups using his induction theorem.   This allows ones to factor the Cartan matrix of the monoid over a splitting field of positive characteristic as a product $D^TCD$ where $D$ is a full rank nonnegative integer matrix and $C$ is the Cartan matrix of the complex monoid algebra.  This suggests that the characteristic zero theory influences the theory in positive characteristic.  For example, if $p>0$ is a prime not dividing the order of  any maximal subgroup of the monoid, then the decomposition matrix is unimodular, and so nonsingularity in characteristic zero passes to nonsingularity in characteristic $p$.  Let me remark, that in this situation $D$ does not have to be a permutation matrix, even in the case that the complex monoid algebra is semisimple.

I do characterize the situation in which $D$ must be a permutation matrix.  Moreover, I propose a definition of modular characteristic for finite monoids and show that in nonmodular characteristic the decomposition matrix is a permutation matrix and a quiver presentation for the characteristic $p$ basic algebra can be lifted to characteristic $0$.  To show that decomposition matrix techniques have some teeth, I use them to provide a new proof of Glover's Theorem~\cite{Glover} that the monoid of $2\times 2$ matrices over the $p$-element field has infinite representation type over fields of characteristic $p$ and I compute a quiver presentation for the case $p=2$, which yields a string algebra first studied by Gelfand and Ponomarev~\cite{lorentz}.

Brauer characters have already been developed in slightly different ways by different authors~\cite{guralnick,PutchaBrauer}.  The definition used here is essentially equivalent to that of~\cite{PutchaBrauer}, whereas  the definition used in~\cite{guralnick} is what is called here the Brauer lift of a Brauer character, as it is a virtual character lifting the more natural definition of Brauer character.   We develop the properties of Brauer characters more fully than in~\cite{guralnick,PutchaBrauer}.
 It is shown that nonsingularity of the Cartan matrix in characteristic zero implies nonsingularity in characteristic $p$  if the ordinary character of the lift to characteristic zero of a finitely generated projective module from characteristic $p$ is determined by its restriction to the $p$-regular elements of the monoid, that is, the elements of the monoid whose period is not divisible by $p$.  Unlike the case of groups, these characters do not have to vanish on the non-$p$-regular elements (even for von Neumann regular monoids).

\section{Preliminaries}
This section gathers together essentially well-known material.

\subsection{Orders and finite dimensional algebras}

If $A$ is a finite dimensional algebra over a field $K$, then the sets of isomorphism classes of projective indecomposable (left) $A$-modules and of simple $A$-modules are in bijection. Choose $P_1,\ldots, P_n$ representing the isomorphism classes of projective indecomposable modules and let $S_i=P_i/\rad P_i$ be the corresponding simple module.  The \emph{Cartan matrix} of $A$ is the $n\times n$-matrix $C(A)$ with $C(A)_{ij}$ the multiplicity of $S_i$ as a composition factor of $P_j$.   Note that if $K$ is a splitting field for $A$, then $C(A)_{ij}=\dim_K \Hom_A(P_i,P_j) =\dim_K e_iAe_j$ where $e_i,e_j$ are primitive idempotents with $P_i\cong Ae_i$ and $P_j\cong Ae_j$; for nonsplitting fields, $C(A)_{ij}$ is the length of $e_iAe_j$ as an $e_iAe_i$-module.

Alternatively, $C(A)$ is the matrix of the Cartan homomorphism $c\colon K_0(A)\to G_0(A)$ with respect to the bases for $K_0(A)$ and $G_0(A)$ consisting of the isomorphism classes of projective indecomposable modules and simple modules, respectively.  Recall that $K_0(A)$ is the Grothendieck group of the monoid of isomorphism classes of finitely generated projective $A$-modules under direct sum, whereas $G_0(A)$ is the Grothendieck group of the monoid of isomorphism classes of finitely generated $A$-modules (under direct sum) factored by the congruence generated by putting $[M]=[L]+[N]$ if there is an exact sequence $0\to L\to M\to N\to 0$.

If $A$ is semisimple, then $C(A)$ is the identity matrix.  Note that $C(A^{op}) = C(A)^T$ in the case $K$ is a splitting field for $A$ as $\Hom_{A^{op}}(e_jA,e_iA)\cong e_iAe_j\cong \Hom_A(Ae_i,Ae_j)$ for primitive idempotents $e_i,e_j$.  The \emph{Cartan determinant} of $A$ is $\cd (A) = \det C(A)$.   If $A$ has finite global dimension, then $\cd (A)=\pm 1$, and the famous Cartan determinant conjecture says that $\cd (A)$ should then be $1$.

If $K$ is a splitting field for $A$, then there is perfect pairing $\langle \cdot,\cdot\rangle\colon K_0(A)\times G_0(A)\to \mathbb Z$ defined by $\langle [P],[M]\rangle = \dim_K\Hom_A(P,M)$ for a finitely generated projective module $P$ and a finitely generated $A$-module $M$.  The isomorphism classes of projective indecomposable modules and of simple modules are dual bases with respect to this pairing, that is, $\langle [P_i],[S_j]\rangle=\delta_{ij}$.

The following generalization of the work of Brauer seems to be folklore (cf.~\cite{WebsterMO}), but I had difficulty finding a reference with proofs at this level of generality, as most texts treat only the case of orders in semisimple algebras. Therefore, I include more details than perhaps is strictly speaking necessary.   Let $\mathcal O$ be a discrete valuation ring with maximal ideal $\mathfrak m$, field of fractions $F$ and residue field $k=\mathcal O/\mathfrak m$.  We will typically be in the case that $F$ has characteristic $0$, the residue field $k$ has characteristic $p>0$ and $\mathcal O$ is complete.

By an \emph{$\mathcal O$-lattice}, we mean a finitely generated free $\mathcal O$-module.   If $V$ is a finite dimensional $F$-vector space, then a \emph{full lattice} $L$ in $V$ is an $\mathcal O$-lattice $L\leq M$ containing a basis for $V$ (or equivalently, spanning $V$ over $F$).  In this case, $V\cong F\otimes_{\mathcal O} L$. Note that any $\mathcal O$-basis for $L$ is an $F$-basis for $V$.   Every finite dimensional $F$-vector space $V$ contains a full lattice.

If $A$ is a finite dimensional $F$-algebra, then an $\mathcal O$-subalgebra $R\leq A$ is an \emph{order} in $A$ if it is a full lattice. Much of the literature focuses on orders in semisimple algebras, but we shall need the general case.  Every finite dimensional $F$-algebra has an order.  We will be interested in the case where $M$ is a finite monoid, in which case $\mathcal OM$ is an order in the monoid algebra $FM$.

Let $R$ be a fixed order in the finite dimensional $F$-algebra $A$.
Note that $\ov R=R/\mathfrak mR\cong k\otimes_{\mathcal O} R$ is a finite dimensional $k$-algebra of $k$-dimension $\dim_F A$.

If we assume that $\mathcal O$ is complete, standard arguments~\cite{CurtisReinerI,Thevenaz,LamBook} show that $\rad^r R\subseteq\mathfrak mR\subseteq \rad R$ for some $r\geq 2$, $R/\rad R\cong \ov R/\rad \ov R$ is a finite dimensional semisimple $k$-algebra and idempotents lift modulo $\rad R$ and modulo $\mathfrak mR$. It follows that $R$ is a semiperfect ring (and so the regular module $R$ has a Krull-Schmidt-Azumaya decomposition) and the natural map $K_0(R)\to K_0(\ov R)$ given by $[P]\mapsto [k\otimes_{\mathcal O} P]$ is an isomorphism.    Therefore, we shall often identify $K_0(R)$ with $K_0(\ov R)$. Note that $k\otimes_{\mathcal O}P \cong P/\mathfrak mP$.   Moreover, the simple $R$-modules are precisely the simple $\ov R$-modules, inflated to $R$, and if $P$ is the projective cover of a simple $\ov R$-module $S$, inflated to an $R$-module, then $P/\mathfrak m P$ is the projective cover of $S$ as an $\ov R$-module.

 From now one we shall assume that $F$ is a splitting field for $A$ and $k$ is a splitting field for $\ov R$.
The analogue of Brauer's cde-triangle in this setup (with $\mathcal O$ complete) is, in fact,  a square when $A$ is not assumed semisimple (cf.~\cite{WebsterMO}).

First note that $A\cong F\otimes_{\mathcal O} R$ as $F$-algebras, and so we have a well-defined homomorphism $e\colon K_0(\ov R)\to K_0(A)$ given by $[\ov P]\mapsto [F\otimes_{\mathcal O} P]$ where $P$ is a projective $R$-module with $P/\mathfrak mP\cong \ov P$.
Our next goal is to define a decomposition homomorphism $d\colon G_0(A)\to G_0(\ov R)$.  The Brauer-Nesbitt approach for group algebras works \textit{mutatis mutandis}.  We sketch the proof for completeness, claiming no originality.
For the moment we just assume that $\mathcal O$ is a discrete valuation ring.  We shall reimpose completeness shortly.

We call a finitely generated $R$-module $M$ an \emph{$R$-lattice} if it is an $\mathcal O$-lattice when viewed as an $\mathcal O$-module.  Notice that every finitely generated projective $R$-module is an $R$-lattice since $R$ is an $\mathcal O$-lattice.   We say that an $R$-lattice $M$ is an \emph{$R$-form} of an $A$-module $V$ if $V\cong F\otimes_{\mathcal O} M$.  If $V$ is an $A$-module and $M$ is an $R$-submodule, which is a full $\mathcal O$-lattice, then we call $M$ a \emph{full} $R$-lattice in $V$.  If $M$ is a full $R$-lattice in $V$, then $M$ is an $\mathcal O$-form of $V$ as the natural map $F\otimes_{\mathcal O}M\to V$ is an isomorphism.  Conversely, every $\mathcal O$-form of $V$ is isomorphic to a full $R$-lattice in $V$ as $1\otimes M$ is a full $R$-lattice in $F\otimes_{\mathcal O} M$ for an $R$-lattice $M$.

 Note that every finite dimensional $A$-module $V$ has an $R$-form.  Indeed, fix a full $\mathcal O$-lattice $L$ in $V$. Then $M=RL$ is an $R$-submodule of $V$.  Moreover, it is finitely generated as an $R$-module since if $b_1,\ldots, b_r$ form an $\mathcal O$-basis for $L$, then they generate $M$ as an $R$-module.  Since $L\subseteq M$ and $L$ is a full $\mathcal O$-lattice, we deduce that $M$ is a full $R$-lattice in $V$.

\begin{Prop}[Brauer-Nesbitt]\label{p:BN}
Let $\mathcal O$ be a discrete valuation ring with maximal ideal $\mathfrak m$, field of fractions $F$ and residue field $k=\mathcal O/\mathfrak m$. Let $A$ be a finite dimensional $F$-algebra and $R$ an order in $A$. Let $\ov R=R/\mathfrak mR$.   Then there is a homomorphism $d\colon G_0(A)\to G_0(\ov R)$ defined by $d([V]) = [k\otimes_{\mathcal O} M]$ where $M$ is an $R$-form of $V$.
\end{Prop}
\begin{proof}
The first step is to show that if $L$ and $M$ are full $R$-lattices in $V$, then in $G_0(\ov R)$ we have $[L/\mathfrak mL]=[M/\mathfrak mM]$.  Note that $L+M$ is a full $R$-lattice in $V$.  Thus, without loss of generality we may assume $L\leq M$.  Note that $M/L$ is a finitely generated torsion $\mathcal O$-module since $L$ has full rank in $M$.  Thus $M/L$ is a finite length $\mathcal O$-module and hence a finite length $R$-module.  A simple induction on length reduces to the case $L$ is a maximal $R$-submodule of $M$.  In this case $M/L$ is a simple $R$-module.  But since $\mathfrak mR\subseteq \rad R$, we in fact have that $\mathfrak m(M/L)=0$, whence $M/L$ is an $\ov R$-module and $\mathfrak mM\subseteq L$.  Thus we have
\[\mathfrak mL\leq \mathfrak mM\leq L\leq M\]
leading to exact sequences of $\ov R$-modules
\begin{gather*}
0\to L/\mathfrak mM\to M/\mathfrak mM\to M/L\to 0\\ 0\to \mathfrak mM/\mathfrak mL\to L/\mathfrak mL\to L/\mathfrak mM\to 0.
\end{gather*}
Since $\mathcal O$ is a discrete valuation ring, we have $\mathfrak m = (\pi)$ for some element $\pi\in R$.  Then $M/L\cong \pi M/\pi L=\mathfrak mM/\mathfrak mL$ via $m+L\mapsto \pi m+\pi L$.  Thus in $G_0(\ov R)$, we have  $[M/\mathfrak mM] = [L/\mathfrak mM]+[M/L] = [L/\mathfrak mM]+[\mathfrak mM/\mathfrak mL]= [L/\mathfrak m L]$, as required.

To see that $d$ induces a well-defined group homomorphism, let $0\to V_1\xrightarrow{\psi} V_2\xrightarrow{\p} V_3\to 0$ be an exact sequence of finitely generated $A$-modules.  Let $M_2$ be a full $R$-lattice in $V_2$.  Then $M_3=\p(M_2)$ is a full $R$-lattice in $V_3$.  Let $M_1=\psi\inv(\ker \p|_{M_2})$.  Then we have
\begin{equation}\label{eq:rforms}
0\to M_1\xrightarrow{\psi} M_2\xrightarrow{\p} M_3\to 0
\end{equation}
 is an exact sequence of $R$-modules.  Since tensoring with $F$ is exact, $\dim_F V_1=\dim_F F\otimes_{\mathcal O}M_1$, and so we deduce that $M_1$ is a full $R$-lattice in $V_1$.  Also, since $M_3$ is a free $\mathcal O$-module, the sequence \eqref{eq:rforms} splits over $\mathcal O$, and hence tensoring it with $k$ over $\mathcal O$ yields an exact sequence $0\to k\otimes_{\mathcal O} M_1\to k\otimes_{\mathcal O} M_2\to k\otimes_{\mathcal O} M_3\to 0$.  It now follows that $d$ is a well-defined group homomorphism.
\end{proof}

The homomorphism $d$ is called the \emph{decomposition map}.   Note that if $V$ is a finitely generated $A$-module and $d([V])=0$, then $V=0$.  Indeed, if $M$ is an $\mathcal O$-form of $V$, then $0=d([V]) = M/\mathfrak mM$, and so $M=0$ by Nakayama's lemma as $M$ is finitely generated as an $\mathcal O$-module.  Thus $V=0$.

We are now ready for the analogue of Brauer's result for an order in a not necessarily semisimple algebra.

\begin{Thm}\label{t:brauer.square}
Let $\mathcal O$ be a complete discrete valuation ring with field of fractions $F$ and residue field $k$.  Let $R$ be an $\mathcal O$-order in a finite dimensional $F$-algebra $A$.  Let $\ov R=R/\mathfrak mR$.  Then we have a commutative diagram
\[\begin{tikzcd}
    K_0(A)\ar{r}{C} & G_0(A)\ar{d}{d}\\ K_0(\ov R)\ar{r}{c}\ar{u}{e} & G_0(\ov R)
  \end{tikzcd}\]
  where $C,c$ are the Cartan homomorphisms, $e([P/\mathfrak mP]) = [F\otimes_{\mathcal O} P]$ and $d$ is the decomposition map.  Moreover, if $F$ and $k$ are splitting fields for $A$ and $\ov R$, respectively, then $e$ is the adjoint of $d$ with respect to the usual perfect pairings.
\end{Thm}
\begin{proof}
The commutativity of the diagram is straightforward since if $P$ is a projective $R$-module, then $P$ is an $R$-form of $F\otimes_{\mathcal O}P$ and so $dCe([P/\mathfrak mP])=dC([F\otimes_{\mathcal O}P]) = d([F\otimes_{\mathcal O} P]) = [P/\mathfrak mP] = c([P/\mathfrak mP])$.

For the final statement, we observe that if $M,N$ are $R$-lattices, then $\Hom_R(M,N)$ is a full $\mathcal O$-lattice in $\Hom_A(F\otimes_{\mathcal O}M,F\otimes_{\mathcal O} N)$.  The easiest way to see this is to note that if $M$ and $N$ have ranks $m,n$ respectively, then, choosing $\mathcal O$-bases for $M$ and $N$, we can identify $\Hom_R(M,N)$ with $n\times m$-matrices over $\mathcal O$ intertwining the $R$-actions and $\Hom_A(F\otimes_{\mathcal O}M,F\otimes_{\mathcal O} N)$ with $n\times m$-matrices over $F$ intertwining the $R$-actions, and the former is clearly a full $\mathcal O$-lattice in the latter.  We conclude that if $P$ is a finitely generated projective $R$-module and $V$ is a finitely generated $A$-module with $R$-form $M$, then
\[\dim_F(e([P/\mathfrak mP]),V) = \dim_F \Hom_A(F\otimes_{\mathcal O} P,V) = \dim_F F\otimes_{\mathcal O}\Hom_R(P,M).\]  On the other hand, the natural homomorphism \[\Hom_R(P,M)\to \Hom_{\ov R}(P/\mathfrak mP,M/\mathfrak mM)\] is surjective (since $P$ is projective) and has kernel $\Hom_R(P,\mathfrak mM)$. Choose $\pi\in R$ with $\mathfrak m = (\pi)$.  Then evidently, $\Hom_R(P,\mathfrak mM) =\Hom_R(P,\pi M)= \pi\Hom(P,M)$.  Therefore, \[\Hom_{\ov R}(P/\mathfrak mP,M/\mathfrak mM)\cong \Hom_R(P,M)/\pi\Hom_R(P,M)\cong k\otimes_{\mathcal O} \Hom_R(P,M).\]  We deduce that $\langle e([P/\mathfrak mP]),[V]\rangle = \dim_F (F\otimes_{\mathcal O}\Hom_R(P,M))=\dim_k (k\otimes_{\mathcal O}\Hom_R(P,M)) = \dim_k \Hom_{\ov R}(P/\mathfrak mP,M/\mathfrak mM) = \langle [P/\mathfrak mP], d([M])\rangle$, as all these numbers are the rank of $\Hom_R(P,M)$ as a free $\mathcal O$-module.  This completes the proof.
\end{proof}

\subsection{Monoids and their algebras}
References for the theory of finite monoids are~\cite{Arbib,qtheor} or~\cite[Chapter~1]{repbook}.  The reader may also consult~\cite{CP} for the algebraic theory of semigroups.
Let $M$ be a finite monoid.  An ideal in a monoid $M$ is a subset $I$ with $MIM\subseteq I$. Left and right ideals are defined analogously.
Two elements $m,n\in M$ are $\mathscr J$-equivalent~\cite{Green}, if they generate the same principal ideal, that is, $MmM=MnM$.  They are $\mathscr L$-equivalent if $Mm=Mn$ and $\mathscr R$-equivalent if $mM=nM$.  The $\mathscr J$-class (respectively, $\mathscr L$-class and $\mathscr R$-class) of $m$ is denoted $J_m$ (respectively, $L_m$ and $R_m$).  An element $m\in M$ is (von Neumann) \emph{regular} if $m=mnm$ for some $n\in M$.  The monoid $M$ is \emph{regular} if all its elements are regular.  Note that $m$ is regular if and only if $J_m$ contains an idempotent, if and only if $L_m$ (or equivalently, $R_m$) contains an idempotent.   A $\mathscr J$-class containing an idempotent is called a \emph{regular} $\mathscr J$-class.  We denote by $E(M)$ the set of idempotents of a monoid $M$.

If $e\in M$ is an idempotent, then $eMe$ is a monoid with identity $e$.  The group of units of $eMe$ is denoted $G_e$ and called the \emph{maximal subgroup} of $M$ at $e$.  Note that if $MeM=MfM$, then $eMe\cong fMf$ and $G_e\cong G_f$.  Also note that $G_e$ acts freely on the right of $L_e$ and on the left of $R_e$ by multiplication.

A \emph{principal series} for $M$ is an unrefinable series
\begin{equation}\label{eq:principal.series}
\emptyset=I_0\subsetneq I_1\subsetneq\cdots \subsetneq I_n=M
\end{equation}
 of ideals of $M$.  It is well known that any two principal series have the same length and that the set differences $I_k\setminus I_{k-1}$ run through all the $\mathscr J$-classes of $M$ exactly once.

Munn~\cite{Munn1} and Ponizovskii~\cite{Poni}, independently, constructed the irreducible representations of a finite monoid, building on earlier work of Clifford~\cite{Clifford2}. We summarize the module theoretic approach to their work~\cite{gmsrep,repbook}.
Let $e\in M$ be an idempotent and let $I$ be any ideal of $M$ with $Ie=Me\setminus L_e$ (e.g., $I=I_{k-1}$ in a principal series \eqref{eq:principal.series} with $J_e=I_k\setminus I_{k-1}$).  If $K$ is a field, then $e(KM/KI)e\cong KG_e$ and $(KM/KI)e$ is a $KM$-$KG_e$-bimodule (free as a right $KG_e$-module) with basis the cosets of the elements of $L_e$.  Hence, we may identify $(KM/KI)e$ with $KL_e$.  With this identification the right $G_e$-action is just via multiplication, whereas the left action of $m\in M$ is given by
\begin{equation}\label{eq:schutz}
m\cdot \ell = \begin{cases} m\ell, &\text{if}\ m\ell\in L_e\\ 0, & \text{else}\end{cases}
\end{equation}
 for $\ell\in L_e$.  If $V$ is a simple $KG_e$-module, then $\Ind_{G_e}(V) = KL_e\otimes_{KG_e} V$ is an indecomposable $KM$-module with simple top $\Ind_{G_e}(V)/\rad(\Ind_{G_e}(V))$.    If $e_1,\ldots, e_t$ are idempotent representatives  of the regular $\mathscr J$-classes of $M$, then, for each simple $KM$-module $S$, there is a unique idempotent $e_i$ from this list (called the \emph{apex} of $S$) with $e_iS\neq 0$  and $e_jS=0$ whenever $Me_iM\nsubseteq Me_jM$.  Moreover, $e_iS$ is a simple $KG_{e_i}$-module and \[S\cong \Ind_{G_{e_i}}(e_iS)/\rad(\Ind_{G_{e_i}}(e_iS)).\]   See~\cite[Theorem~5.5]{repbook} for details.  Notice that, in particular, there is a bijection between irreducible representations of $M$ and the combined set of irreducible representations of $G_{e_1},\ldots, G_{e_t}$.

The character theory of monoids was developed independently by McAlister~\cite{McAlisterCharacter} and Rhodes and Zalcstein~\cite{RhodesZalc} in characteristic zero (see also~\cite[Chapter~7]{repbook}).  The theory over a general field was developed in~\cite{ourcharacter} (some less detailed results can also be found in~\cite{guralnick}).  The reader is referred to these references for details.

We retain the previous notation. If $m\in M$, then there are least integers $i,k$ with $i\geq 0$, $k>0$ and $m^i=m^{i+k}$.  One calls $i$ the \emph{index} of $m$ and $k$ the \emph{period} of $m$.  There is a unique idempotent in the subsemigroup generated by $m$, and it is denoted $m^{\omega}$. If $m$ has index $i$ and period $k$, then this idempotent is $m^r$ where $r\geq i$ and $r\equiv 0\bmod k$; for instance $m^{\omega}=m^{isk}$ for any  $s>0$.  We denote $mm^{\omega}$ by $m^{\omega+1}$.  Note that $m^{\omega+1}$ belongs to the maximal subgroup $G_{m^{\omega}}$ and has order the period of $m$.

Two elements $m_1,m_2\in M$ are \emph{generalized conjugates}, written $m_1\sim m_2$, if there exist $x,x'\in M$ with $xx'x=x$, $x'xx'=x'$, $x'x=m_1^{\omega}$, $xx'=m_2^{\omega}$ and $xm_1^{\omega+1}x'=m_2^{\omega+1}$.  It is straightforward to check that generalized conjugacy is the same as conjugacy if $M$ is a group.

Generalized conjugacy is an equivalence relation and it is the least equivalence relation on $M$ such that $mm'\sim m'm$ and $m\sim m^{\omega+1}$ for all $m,m'\in M$.  Each generalized conjugacy class intersects exactly one of the maximal subgroups $G_{e_1},\ldots, G_{e_t}$, and it intersects that maximal subgroup in a conjugacy class.  Thus we may choose a set of generalized conjugacy class representatives by choosing a representative of each conjugacy class of $G_{e_1},\ldots, G_{e_t}$.  In particular, if follows that the number of generalized conjugacy classes is the number of simple $KM$-modules for any splitting field $K$ of $M$ whose characteristic does not divide the order of any maximal subgroup by the Clifford-Munn-Ponizovskii theorem.  We remark that if $M$ is regular, then generalized conjugacy is simply the least equivalence relation on $M$ such that $mm'\sim m'm$ for all $m\in M$ by a result from~\cite{Mazorchuk}.

If $p>0$ is a prime, then an element of a group is called \emph{$p$-regular} it its order is prime to $p$.  If $p=0$, then all elements are considered $p$-regular.  If $g$ is an element of a group, then $g$ can be uniquely factored in the form $g=g_pg_{p'}$ where $g_p$ has $p$-power order, $g_{p'}$ is $p$-regular and $g_pg_{p'}=g_{p'}g_p$.  If $m$ is an element of a monoid, then we define $m$ to be \emph{$p$-regular} if $m^{\omega+1}$ is $p$-regular (in the group $G_{m^{\omega}})$, or, equivalently, $p$ does not divide the period of $m$.  If $m\sim m'$, then $m$ is $p$-regular if and only if $m'$ is $p$-regular, and so we may speak of $p$-regular generalized conjugacy classes.  Each $p$-regular generalized conjugacy class intersects a unique maximal subgroup $G_{e_i}$ from $G_{e_1},\ldots, G_{e_n}$, and it intersects it in a $p$-regular conjugacy class.  In particular, it follows from the Clifford-Munn-Ponizovskii theorem and Brauer's counting of the number of irreducible modular representations of a finite group~\cite[Theorem~9.3.6]{Webbbook} that if $K$ is a splitting field for $M$ of characteristic $p$, then the number of simple $KM$-modules is the number of $p$-regular generalized conjugacy classes of $M$.

If $K$ is a field, then a $K$-valued \emph{class function} on $M$ is a mapping $f\colon M\to K$ which is constant on generalized conjugacy classes.  If $V$ is a finitely generated $KM$-module, then the \emph{character} $\chi_V\colon M\to K$ is the mapping given by setting $\chi_V(m)$ equal to the trace of the linear operator on $V$ given by multiplication by $m$.  An \emph{irreducible} character is the character of a simple $KM$-module.   The character of any module is a class function and if $K$ is a splitting field of characteristic $0$, then the irreducible characters of $M$ form a basis for the space of $K$-valued class functions.  If $K$ has characteristic $p>0$ and $\chi$ is a character of a $KM$-module, then $\chi(m) = \chi(m^{\omega+1}_{p'})$ for all $m\in M$ (retaining the previous notation).  Moreover, if $K$ is a splitting field then the irreducible characters form a basis for the class functions on $M$ satisfying the additional condition $f(m) = f(m^{\omega+1}_{p'})$.  In particular, the characters of $M$ over $K$ are linearly independent (this is true for any finite dimensional algebra over a splitting field) and determined by their values on $p$-regular conjugacy classes.

We shall also need McAlister's characterization of virtual characters.  Recall that a virtual character is a difference of characters.  The mapping $[V]\mapsto \chi_V$ provides an isomorphism between $G_0(KM)$ and the ring of virtual characters for any field $K$ (by the results of~\cite{ourcharacter} or~\cite{guralnick}).  McAlister showed that if $K$ is a splitting field of characteristic $0$, then a class function $f$ is a virtual character if and only if $f|_{G_e}$ is a virtual character for each maximal subgroup $G_e$ of $M$ (cf.~\cite[Corollary~7.16]{repbook}).  This, combined with Brauer's characterization of virtual characters for a group, provides an effective tool to determine if a class function is a virtual character.

\subsection{Quivers and bound quivers}
A \emph{quiver} $Q$ is a finite directed multigraph (possibly with loop edges).   The vertex set of $Q$ is denoted $Q_0$ and the edge set $Q_1$. More generally, $Q_n$ will denote the set of (directed) paths of length $n$ with $n\geq 0$ (where the distinction between empty paths and vertices is ignored).  Basic references about quivers and bound quivers are~\cite{assem,benson}.

The \emph{path algebra} $RQ$ of $Q$ over a commutative ring $R$ has $R$-basis consisting of all paths in $Q$, including an empty path $\varepsilon_v$ at each vertex, with the product induced by concatenation (where undefined concatenations are set equal to zero).  Here paths are composed from right to left: so if $p\colon v\to w$ and $q\colon w\to z$ are paths, then their composition is denoted $qp$.  This convention is because we work with left modules. Note that $RQ=\bigoplus_{n\geq 0} RQ_n$. The identity of $RQ$ is the sum of all empty paths.
The arrow ideal $J$ is the ideal of $RQ$ generated by all edges, that is, $J=\bigoplus_{n\geq 1} RQ_n$.

A finite dimensional $K$-algebra $A$ is \emph{split} if each simple $A$-module is absolutely simple or, equivalently, $A/\rad(A)$ is isomorphic to a direct product of matrix algebras over $K$.  The algebra $A$ is \emph{split basic} if $A/\rad(A)\cong K^n$ for some $n\geq 1$. (Generally speaking, $A$ is \emph{basic} if $A/\rad(A)$ is a product of division algebras.)  Every split $K$-algebra is Morita equivalent to a unique (up to isomorphism) split basic $K$-algebra, and this algebra can be described in terms of a bound quiver.

Suppose that $A$ is a finite dimensional algebra and $1=\varepsilon_1+\cdots+\varepsilon_s$ is a decomposition into orthogonal primitive idempotents.  Note that if $e,f$ are primitive idempotents of $A$, then $Ae\cong Af$ if and only if $e,f$ are conjugate.  Without loss of generality, we may assume that $\varepsilon_1,\ldots, \varepsilon_k$ are pairwise nonconjugate and represent all the conjugacy classes of primitive idempotents of $A$.  Note that $k=s$ if and only if $A$ is basic.  Let $e=e_1+\cdots +e_k$.  Then $e$ is a full idempotent, i.e., $A=AeA$,  and so $B=eAe$ is Morita equivalent to $A$.  The algebra $B$ is basic and is called the \emph{basic algebra} of $A$ and $e$ is called a \emph{basic idempotent}.  Note that $A$ is split if and only if $B$ is split.

If $K$ is a field, an ideal $I$ of $KQ$ is called \emph{admissible} if there exists $n\geq 2$ with $J^n\subseteq I\subseteq J^2$.   A pair $(Q,I)$ consisting of a quiver $Q$ and an admissible ideal $I$ of $KQ$ is called a \emph{bound quiver}.  If $(Q,I)$ is a bound quiver, then $KQ/I$ is a split basic finite dimensional algebra with Jacobson radical $J/I$. Conversely, if $A$ is a split basic $K$-algebra, there is a unique (up-to-isomorphism) quiver $Q(A)$ such that $A\cong KQ(A)/I$ for some admissible ideal $I$ (the ideal $I$ is not unique).  The bound quiver $(Q,I)$ is sometimes called a \emph{quiver presentation} for $A$.

The vertices of $Q(A)$ are the (isomorphism classes of) simple $A$-modules. If $S,S'$ are simple $A$-modules, then the number of directed edges $[S]\to [S']$ in $Q(A)$ is $\dim_K\Ext^1_A(S,S')$.  To explain how the admissible ideal $I$ is obtained, fix a complete set of orthogonal primitive idempotents $E$ for $A$, and let $e_S\in E$ be the idempotent with $S\cong (A/\rad(A))e_S$. One has that \[\Ext^1_A(S,S')\cong e_{S'}[\rad(A)/\rad^2(A)]e_S\] as $K$-vector spaces~\cite{assem,benson}.  One can define a homomorphism $KQ(A)\to A$  by sending the empty path at $S$ to $e_S$ and by choosing a bijection of the set of  arrows $S\to S'$ with a subset of $e_{S'}\rad(A)e_S$ mapping to a basis of $e_{S'}[\rad(A)/\rad^2(A)]e_S$. The map is extended to paths of length greater than one in the obvious way.  Such a homomorphism is automatically surjective and the kernel is an admissible ideal.  If $(Q,I)$ is a bound quiver, then the quiver of $KQ/I$ is isomorphic to $Q$ with the vertex $v$ corresponding to the simple module $[(KQ/I)/(J/I)]\varepsilon_v$.

More generally, we define the quiver of a split $K$-algebra to be that of its basic algebra, and similarly for quiver presentations.  Note that if $e$ is a basic idempotent of $A$, then $\rad(eAe)=e\rad(A)e$ and $\rad^2(eAe) = e\rad(A)e\rad(A)e=e\rad(A)AeA\rad(A)e=e\rad^2(A)e$, since $AeA=A$.  It follows that if $e=e_1+\cdots+e_k$ is a decomposition into orthogonal primitive idempotents, then $e_i[\rad(A)/\rad^2(A)]e_j\cong e_i[\rad(eAe)/\rad^2(eAe)]e_j$, a fact that we shall use later without comment.

\section{Modular representation theory of monoids}
This section sketches the basics of monoid modular representation theory.  This does not seem to be a well-developed subject, and mostly there are
preliminary results concerning Brauer characters~\cite{guralnick,PutchaBrauer} and characters~\cite{guralnick,ourcharacter}.

Recall that Brauer's theorem~\cite[Theorem~17.1]{CurtisReinerI}) says that if $G$ is a group of exponent $n$ and $K$ is a field containing a primitive $n^{th}$-root of unity, then $K$ is a splitting field for $KG$.  Brauer's theorem almost immediately generalizes to monoids, thanks to Clifford-Munn-Ponizovskii theory.  We give a module theoretic treatment, although this is obvious from matrix theoretic treatments (cf.~\cite{RhodesZalc}).

\begin{Prop}\label{p:radical.extension}
Let $M$ be a finite monoid and $K$ a field.  Let $F\leq K$ be the prime field.  Let $L/K$ be a field extension.  Identifying $LM$ with $L\otimes_K KM$ and $KM$ with $K\otimes_F FM$, the following hold.
\begin{enumerate}
\item $\rad KM = K\otimes_F \rad FM$.
\item $\rad LM= L\otimes_K \rad KM$.
\item If $V$ is a finitely generated $KM$-module, then $\rad (L\otimes_K V)= L\otimes_K \rad V$.
\end{enumerate}
\end{Prop}
\begin{proof}
Clearly $I=K\otimes_F\rad FM$ is a nilpotent ideal and hence contained in $\rad KM$.  On the other hand, $KM/I\cong K\otimes_F (FM/\rad FM)$.  Since $F$ is a perfect field, every semisimple $F$-algebra is separable and hence remains semisimple under base extension.  Thus $K\otimes_F (FM/\rad FM)$ is semisimple, and so $\rad KM\subseteq K\otimes_F\rad FM$.  This proves the first item.

The second follows from the first applied to $K$ and $L$, as $L\otimes_K\rad KM = L\otimes_K (K\otimes_F \rad FM)=L\otimes_F \rad FM=\rad LM$.

The final item follows because $\rad(L\otimes_K V) = (\rad LM)(L\otimes_K V) = (L\otimes_K \rad KM)(L\otimes_K V) = L\otimes_K \rad V$.
\end{proof}

We shall use throughout the following well-known fact on how tensor products interact with extension of scalars.

\begin{Lemma}\label{l:commute.with.ext}
Let $S$ be a commutative ring and $R$ a subring.  Let $A,B$ be $R$-algebras, $M$ an $A$-$B$-bimodule (with the same $R$-action on both sides) and $N$ a $B$-module.  Then $S\otimes_R(M\otimes_B N)\cong (S\otimes_R M)\otimes_{S\otimes_R B}(S\otimes_R N)$ as $S\otimes_R A$-modules.
\end{Lemma}
\begin{proof}[Sketch of proof]
One checks both of these functors represent $\Hom_A(M\otimes_B N,\Res(-))$, where $\Res$ is restriction of scalars from $S\otimes_R A$ to $A$, and applies the Yoneda lemma.
\end{proof}

The following is surely folklore.

\begin{Thm}\label{t:Brauer+CMP}
Let $M$ be a finite monoid and $K$ a field.  Then $K$ is a splitting field for $M$ if and only if it is a splitting field for each maximal subgroup of $M$.  In particular, if $n$ the least common multiple of the periods of the elements of $M$ and $K$ is a field containing a primitive $n^{th}$-root of unity, then $K$ is a splitting field for $M$.
\end{Thm}
\begin{proof}
First note that if $K$ is a splitting field for a finite dimensional $K$-algebra $A$, then it is a splitting field for any quotient and any corner $eAe$ with $e=e^2$. In particular, if $K$ is a splitting field for $M$ and $e\in E(M)$, then $K$ is a splitting field for $KG_e\cong eKMe/kI$ where $I$ is the ideal $eMe\setminus G_e$ of $eMe$.  Suppose now that $K$ is a splitting field for each maximal subgroup of $M$, and let $e\in M$ be an idempotent.
Our assumption implies that $K$ is a splitting field for $KG_e$, and so each $KG_e$-module is absolutely simple.  By the Clifford-Munn-Ponizovskii theorem~\cite[Theorem~5.5]{repbook}, we must show that if $V$ a simple $KG_e$-module, then the $KM$-module $\Ind_{G_e}(V)/\rad (\Ind_{G_e})(V)$ is absolutely simple. Let $F/K$ be an extension field. Lemma~\ref{l:commute.with.ext} yields  $F\otimes_K\Ind_{G_e}(V)= F\otimes_K (KL_e\otimes_{KG_e} V)\cong FL_e\otimes_{FG_e} (F\otimes_K V)= \Ind_{G_e}(F\otimes_K V)$. Proposition~\ref{p:radical.extension} then tells us that $\rad(F\otimes_K\Ind_{G_e}(V)) = F\otimes_K \rad(\Ind_{G_e}(V))$.  We conclude that $F\otimes_K (\Ind_{G_e}(V)/\rad(\Ind_{G_e}(V)))\cong (F\otimes_K \Ind_{G_e}(V))/(F\otimes_K\rad (\Ind{G_e}(V)))\cong \Ind_{G_e}(F\otimes_K V)/\rad (\Ind_{G_e}(F\otimes_K V))$ and hence is simple by the Clifford-Munn-Ponizovskii theorem~\cite{repbook} (as $F\otimes_K V$ is simple).  This proves the first statement.  The second follows from the first by Brauer's theorem.
\end{proof}

Note that the least common multiple of the periods of the elements of $M$ is the same as the least common multiple of the exponents of its maximal subgroups.

Let $M$ be a finite monoid. For a prime $p>0$, a (complete) $p$-modular system $(F,\mathcal O,k)$ consists of a complete discrete valuation ring $\mathcal O$ of characteristic $0$ with maximal ideal $\mathfrak m$, field of fractions $F$ and residue field $k=\mathcal O/\mathfrak m$ of characteristic $p$.

We say that the $p$-modular system is \emph{splitting} if $F$ is a splitting field for $FM$ and $k$ is a splitting field for $kM$.  One can always find a splitting $p$-modular system for $M$.  For example, if $n$ is the least common multiple of the periods of the elements of $M$, we can adjoin to the $p$-adic rationals $\mathbb Q_p$ a primitive $n^{th}$-root of unity $\omega$ and we can take $\mathcal O$ to be the integral closure of the ring $\mathbb Z_p$ of $p$-adic integers in $F=\mathbb Q_p(\omega)$.  Then $\mathcal O$ is a complete discrete valuation ring with field of fractions $F$,  $\omega\in \mathcal O$ and the maximal ideal $\mathfrak m$ of $\mathcal O$ contains $p\mathbb Z_p$.  Thus $k=\mathcal O/\mathfrak m$ is a finite field of characteristic $p$ containing a primitive $n^{th}$-root of unity $\omega+\mathfrak m$.  It follows that $(F,\mathcal O,k)$ is a splitting $p$-modular system for $M$ by Theorem~\ref{t:Brauer+CMP}.

Fix from now on a splitting $p$-modular system $(F,\mathcal O,k)$ for $M$.  Our first goal will be to show that the decomposition map $d$ is surjective.  This will in turn imply that $e$ is injective, as it is the adjoint to $d$.  We will make use of Brauer's theorem that $d$ is surjective for group rings.

If $R$ is any commutative ring with unit, $e\in E(M)$ and $V$ is an $RM$-module, then $\Res_{G_e}(V)=  eV$ is an $RG_e$-module.  Moreover, $\Res_{G_e}$ is an exact functor from the category of finitely generated $RM$-modules to the category of finitely generated $RG_e$-modules.  There results a ring homomorphism $\Res_{G_e}\colon G_0(RM)\to G_0(RG_e)$.

\begin{Prop}\label{p:decomp.commutes}
If $e\in E(M)$ and $V$ is a finitely generated $FM$-module, then $d([\Res_{G_e}(V)]) = \Res_{G_e}(d([V]))$.
\end{Prop}
\begin{proof}
Let $L$ be a full $\mathcal OM$-lattice in $V$.  Then $eL$ is a full $\mathcal OG_e$-lattice in $eV$.  Also $e(L/\mathfrak mL) \cong eL/e\mathfrak mL=eL/\mathfrak meL$ since restriction is exact.  It follows that $d([\Res_{G_e}(V)]) = d[eV] = [eL/\mathfrak meL] = [e(L/\mathfrak mL)] = \Res_{G_e}(d([V]))$.
\end{proof}

Fix idempotent representatives $e_1,\ldots, e_t$ of the regular $\mathscr J$-classes of $M$.  Then it is shown in~\cite[Theorem~6.5]{repbook} that, for any field $K$, one has that $\Res\colon G_0(KM)\to \prod_{i=1}^tG_0(KG_{e_i})$ given by $\Res([V])_i = \Res_{G_{e_i}}([V])$ is a ring isomorphism.

\begin{Thm}\label{t:decomp.surjective}
Let $(F,\mathcal O,k)$ be a splitting $p$-modular system for $M$.  Then the decomposition map $d\colon G_0(FM)\to G_0(kM)$ is surjective and hence its adjoint $e\colon K_0(kM)\to K_0(FM)$ is injective.
\end{Thm}
\begin{proof}
By Proposition~\ref{p:decomp.commutes} have a commutative diagram
\[\begin{tikzcd}
    G_0(FM)\ar{r}{\Res}\ar{d}[swap]{d} & \prod_{i=1}^t G_0(FG_{e_i})\ar{d}{(d_1,\ldots, d_t)}\\ G_0(kM)\ar{r}{\Res} & \prod_{i=1}^t G_0(kG_{e_i})
  \end{tikzcd}\]
  where $d_i\colon G_0(FG_{e_i})\to G_0(kG_{e_i})$ is the decomposition map.  Since the horizontal maps are isomorphisms and the $d_i$ are surjective by Brauer's theorem~\cite[Corollary~18.14]{CurtisReinerI}, we conclude that $d$ is surjective.
\end{proof}

Let $P_1,\ldots, P_n$ be representatives of the isomorphism classes of projective indecomposable $FM$-modules, with corresponding simples $S_i=P_i/\rad(P_i)$.  Let $Q_1,\ldots, Q_r$ be representatives of the isomorphism classes of projective indecomposable $kM$-modules with corresponding simples $T_i=Q_i/\rad(Q_i)$.  It follows from Theorem~\ref{t:decomp.surjective} that $r\leq n$.  This also can be deduced from the fact that $r$ is the number of $p$-regular generalized conjugacy classes of $M$ and $n$ is the number of generalized conjugacy classes.  Denote by $C(FM)$ and $C(kM)$ the Cartan matrices for $FM$ and $kM$, respectively.  Let $D$ be the $n\times r$ matrix defined by $d([S_i]) = \sum_{j=1}^r d_{ij}[T_j]$.  Notice that $D^T$ is the matrix of $d$ and hence $D$ is the matrix of $e$.  We can then reformulate Theorem~\ref{t:brauer.square} in the following way, generalizing Brauer's result from the group case.

\begin{Thm}\label{t:CartanSquare}
Let $(F,\mathcal O,k)$ be a splitting $p$-modular system for $M$.  Then $C(kM) = D^TC(FM)D$ with $D$ a full rank matrix. In particular, if $FM$ is semisimple, then $C(kM)=D^TD$ is a positive definite symmetric matrix.
\end{Thm}

 We remark that even if $FM$ is semisimple,  $kM$ need not be a symmetric algebra.  For instance, if $M(n,q)$ is the monoid of $n\times n$-matrices over the field of $q$ elements, then $FM(n,q)$ is semisimple for $F$ of characteristic zero~\cite{putchasemisimple,Kovacs,repbook}, but $kM(n,q)$ is symmetric if and only if $q$ is not a power of the characteristic $p>0$ of $k$~\cite{Kovacs,repbook}.

Theorem~\ref{t:CartanSquare} provides a relationship between the Cartan matrix of the algebra of a finite monoid over a splitting field in characteristic $0$ and in characteristic $p>0$.  I originally hoped that nonsingularity of $C(\mathbb CM)$ would imply nonsingularity of $C(KM)$ for all fields $K$, but a counterexample was provided by Eisele~\cite{eisele2023counterexample}.  So instead I propose the following general question.

\begin{Question}\label{conjecture}
What conditions on a finite monoid guarantee that if the Cartan matrix $C(\mathbb CM)$ of the complex algebra $\mathbb CM$ is nonsingular, then the Cartan matrix  $C(KM)$ is nonsingular for any field $K$?
\end{Question}

Later, we shall provide some positive results and conjecture  a condition that encompasses all the current results.

Let us begin with a couple of observations.  First of all the Cartan matrix of a monoid algebra can be singular over every field.  In fact, it can be essentially as complicated as any Cartan matrix can be, as the following construction shows.

\begin{Prop}
Let $A$ be an $n\times n$ nonnegative integer matrix with positive diagonal entries.  Then there is a finite monoid $M$ with $n+2$ simple modules such that \[C(KM) =\begin{bmatrix}A & 0 & 0\\ 0 & 1& 0\\ 0&0&1\end{bmatrix}\]  for all fields $K$.  In particular, $C(KM)$ can be singular over all fields.
\end{Prop}
\begin{proof}
View $A-I$ (which is a nonnegative integer matrix) as the adjacency matrix of a finite quiver $Q$ with $n$ vertices.  More precisely, $Q$ has vertices $1,\ldots, n$ and there are $(A-I)_{ij}$ arrows from vertex $j$ to vertex $i$ in $A$.  Let $J$ be the arrow ideal of $KQ$.  Then $KQ/J^2$ splits over $K$ and has Cartan matrix $A$.  Let $M$ be the submonoid of $KQ/J^2$ consisting of $1,0$, the $\varepsilon_i$ and the edges of $Q$ (where we omit cosets from our notation).  The monoid $M$ depends only on $Q$ and not on $K$.    To avoid confusing the zero of $M$ with the zero of $KM$, we write $z$ for the zero of the monoid.  Let $\varepsilon=\sum (\varepsilon_i-z)\in KM$.  Then there is a surjective homomorphism $\pi\colon KM\to KQ/J^2$ induced by the inclusion of $M$, and $\ker \pi$ is spanned over $K$ by $1-\varepsilon$ and $z$. Moreover,  $1-\varepsilon-z,\varepsilon, z$ are orthogonal central idempotents summing to $1$ with $KM(1-\varepsilon-z) = K(1-\varepsilon-z)$, $KM\varepsilon\cong KQ/J^2$ and $KMz= Kz$.  It follows that $KM\cong KQ/J^2\times K\times K$, and hence has Cartan matrix as described in the statement of the proposition.

In particular, if we take \[A=\begin{bmatrix} 1 & 1\\ 1&1\end{bmatrix},\] then $C(KM)$ is singular for every field $K$.
\end{proof}

Next we observe that for the purposes of the Question~\ref{conjecture}, it is enough to consider the case of a splitting field.

\begin{Prop}\label{p:split.field.enough}
Let $A$ be a finite dimensional algebra over a field $K$ and suppose that $L$ is a splitting field for $A$.  If $C(L\otimes_K A)$ is nonsingular, then $C(A)$ is nonsingular.
\end{Prop}
\begin{proof}
Put $A_L=L\otimes_K A$.
Let $c\colon K_0(A)\to G_0(A)$ and $c'\colon K_0(A_L)\to G_0(A_L)$ be the Cartan maps.  We have homomorphisms $i\colon K_0(A)\to K_0(A_L)$ and $j\colon G_0(A)\to G_0(A_L)$ given by $i([V]) = [L\otimes_K V]$ and $j([V]) = [L\otimes_K V]$.  Then it is immediate that
\[\begin{tikzcd}
    K_0(A)\ar{r}{c}\ar{d}[swap]{i} & G_0(A)\ar{d}{j}\\ K_0(A_L)\ar{r}{c'} & G_0(A_L)
  \end{tikzcd}\]
commutes.

The map $i$ is injective by the Noether-Deuring theorem, while the map $c'$ is injective by hypothesis.  It follows that $jc$, and hence $c$, is injective.  Thus $C(A)$ is nonsingular.
\end{proof}

If $K,L$ are two splitting fields for $M$ of the same characteristic, then $C(KM)=C(LM)$ (up to conjugation by a permutation matrix), and so we can now rephrase our question as asking if $(F,\mathcal O,k)$ is a splitting $p$-modular system for $M$ and $C(FM)$ is nonsingular, under what conditions on $M$ must  $C(kM)$ also be nonsingular.  Inspection of Theorem~\ref{t:CartanSquare} shows that it suffices to prove that $D^T$ is injective on the image of $C(FM)D$.  That is, the question is equivalent to asking when the decomposition map $d$ must be injective on the subgroup of $G_0(FM)$ generated by the classes $[F\otimes_{\mathcal O} P]$ with $P$ a finitely generated projective $\mathcal OM$-module.  
 To understand what this means in a more concrete way, we shall reinterpret things in terms of characters in the next section.

We now show that if $C(\mathbb CM)$ is nonsingular, then $C(kM)$ is nonsingular for fields $k$ whose characteristic does not divide the order of any maximal subgroup of $M$. Note that by Proposition~\ref{p:split.field.enough}, it is enough to consider a splitting field.

\begin{Cor}\label{c:good.char}
Let $M$ be a finite monoid and $p>0$ a prime not dividing the order of any maximal subgroup of $M$.  Let $(F,\mathcal O,k)$ be a splitting $p$-modular system for $M$.  Then the decomposition matrix $D$ is invertible over $\mathbb Z$ and hence $\det C(FM)=\det C(kM)$. Consequently, if $C(\mathbb CM)$ is nonsingular, then $C(KM)$ is nonsingular for any field of characteristic not dividing the order of any maximal subgroup of $M$.
\end{Cor}
\begin{proof}
Under these hypotheses $FM$ and $kM$ have the same number of simple modules since every generalized conjugacy class is $p$-regular.  Thus $d\colon G_0(FM)\to G_0(kM)$ is a surjective homomorphism of finitely generated free abelian groups (by Theorem~\ref{t:decomp.surjective}) of the same rank and hence an isomorphism.  As $D^T$ is the matrix of $d$ with respect to the natural bases, we conclude $\det(D)=\pm 1$.  From $C(kM)=D^TC(FM)D$, we conclude $\det C(kM)=(\det D)^2\det C(FM)=\det C(FM)$, as required.
\end{proof}

In the next section, we shall show that $D$ is, in fact, unipotent upper triangular with respect to an appropriate ordering of the simple modules under the hypotheses of Corollary~\ref{c:good.char}.  But it is not the case, in general,  that $D$ is a permutation matrix,  as the following illustrative example (which also appears in~\cite{PutchaBrauer}) with $FM$ semisimple shows.

\begin{Example}\label{ex:singular.2.2}
 Let $M$ be the submonoid of $M(2,2)$ consisting of all singular $2\times 2$-matrices over $\mathbb F_2$,  together with the identity matrix.  All the maximal subgroups of $M$ are trivial, and so every field is a splitting field.   Moreover, $KM$ is semisimple over all fields $K$ except those of characteristic $2$ (this can be verified directly using the sandwich matrix criterion for semisimplicty~\cite[Theorem~5.21]{repbook} since the only nontrivial sandwich matrix is that of the $\mathscr J$-class of rank one matrices, and this sandwich matrix is
 \begin{equation}\label{eq:sandwich.examp}
 \begin{bmatrix} 0 & 1& 1\\ 1&0&1\\ 1&1&0\end{bmatrix},
 \end{equation}
  which is invertible over all fields except those of characteristic $2$).

 Let us take $(\mathbb Q_2,\mathbb Z_2,\mathbb F_2)$ as our splitting $2$-modular system for $M$.   Let $P=\mathbb P(\mathbb F_2^2)$ be the projective space of lines in $\mathbb F_2^2$.  Then $\mathbb Z_2P$ is a $\mathbb Z_2M$-module where a matrix $A$ acts on a line $\ell$ by
\begin{equation}\label{eq:line.act}
A\odot \ell = \begin{cases}A\ell, & \text{if}\ A\ell \neq \{0\}\\ 0, & \text{else.}\end{cases}
\end{equation}
  Extending the scalars to $\mathbb Q_2$ gives a simple module $S_1=\mathbb Q_2P$.  The other simple $\mathbb Q_2M$-modules are the trivial module $S_2$ and the one-dimensional module $S_3$ corresponding to the determinant representation (sending singular matrices to $0$ and nonsingular matrices to $1$).  The $\mathbb Z_2$-forms of these modules are $\mathbb Z_2$ with the same actions, and so $d([S_2])=[S_2']$ and $d([S_3])=[S_3']$ where $S_2'$ is the trivial $\mathbb F_2M$-module and $S_3'$ is the one-dimensional $\mathbb F_2M$-module given by the determinant representation.  On the other hand, $\mathbb Z_2P/2\mathbb Z_2P\cong \mathbb F_2P$.  But $\mathbb F_2P$ is not simple.  There is an exact sequence $0\to S_3'\to \mathbb F_2P\to S_1'\to 0$ where $S_1'$ is a two-dimensional simple module.  The copy of $S_3'$ is the sum $\ell_1+\ell_2+\ell_3$ of the three lines.  This subspace is invariant since each rank $1$ matrix $A$ annihilates one line and sends the other two lines to the same line $\ell=A\mathbb F_2^2$, and hence $A\odot (\ell_1+\ell_2+\ell_3) = 2\ell=0$.  Thus $d([S_1]) = [S_1']+[S_3']$, and so the decomposition matrix is given by
\[D=\begin{bmatrix} 1 & 0 & 1\\ 0&1&0\\ 0&0&1\end{bmatrix}\]
which is unimodular but not a permutation matrix.

We note that in this case $C(\mathbb Q_2M)$ is the $3\times 3$-identity matrix, while \[C(\mathbb F_2M) = D^TD =\begin{bmatrix}1 & 0 & 1\\ 0 & 1& 0\\ 1 & 0 & 2\end{bmatrix}.\]  This can also be checked directly.   For instance, $\mathbb F_2P$ is the projective cover of $S_1'$, and so the above exact sequence explains the first column.  The simple $S_2'$ is already projective.   The projective cover $Q$ of $S_3'$ is more complicated. One can show there is an exact sequence $0\to \mathbb F_2P\to Q\to S_3'\to 0$, which explains the third column. This also shows that the global dimension of $\mathbb F_2M$ is $2$.    Nonetheless, it seems easier to compute $C(\mathbb F_2M)$ via the decomposition matrix. Note that $\mathbb F_2M$ is not a symmetric algebra (either by~\cite[Chapter~15]{repbook}, or because the above computations shows that it has global dimension $2$).  So there is no reason \textit{a priori} $C(\mathbb F_2M)$ should be symmetric if one doesn't appeal to the characteristic $0$ situation via Theorem~\ref{t:CartanSquare}.   We remark that the basic algebra of $\mathbb F_2M$ is given by the bound quiver $(Q,I)$ where
\[Q=\begin{tikzcd}\bullet & \bullet\arrow[bend right=60,swap]{r}{\alpha} &\bullet\arrow[bend right=60,swap]{l}{\beta}\end{tikzcd}\] and $I$ is generated by $\alpha\beta$, as can be easily checked, and so $\mathbb F_2M$ has finite representation type.
\end{Example}

It is then natural to ask under what conditions the decomposition matrix is given by a permutation matrix, which we take up in the next section.

We end this section by applying the above machinery  to give a new proof of a result of Glover~\cite{Glover} on the representation type of $M(2,p)$.   Like Glover, we shall use the following result of Dickson~\cite[Theorem~3.1]{Dicksonfiniterep}.

\begin{Thm}[Dickson]\label{t:dickson}
Suppose $A$ is a finite dimensional $K$-algebra and $V$ is a finite dimensional indecomposable left $A$-module with endomorphism algebra $R$ (which we view as acting on the right of $V$).  Suppose that $V$ contains isomorphic submodules $U,W$ such that $UR\cap WR=0=U\rad(R)=W\rad(R)$ and that $\Ext_A^1(M,U)=0$.  Then $A$ has infinite representation type.
\end{Thm}

Glover showed that $\mathbb F_pM(2,p)$ has a projective indecomposable module meeting the criterion of Dickson's theorem by describing in detail all the projective indecomposable $\mathbb F_pM(2,p)$-modules.  We shall get away with a less detailed analysis of the projective indecomposables by using the symmetry of the Cartan matrix.  Our argument will work for a class of monoids.

An ideal $I$ of a monoid with zero  is \emph{$0$-minimal} if $I^2\neq \{0\}$ and there are no proper nonzero ideals contained in $I$.  For example, in $M(n,q)$, the ideal of matrices of rank at most $1$ is $0$-minimal since all rank $1$ matrices generate the same principal ideal.  If $I$ is $0$-minimal, then $I\setminus \{0\}$ is a $\mathscr J$-class containing a nonzero idempotent; see~\cite{qtheor} for details.

\begin{Lemma}\label{l:infinite.rep.type.small}
Let $M$ be a monoid with zero and group of units $G$.   Suppose that $k$ is a splitting field for $M$ of characteristic $p>0$.  Assume that $M$ has a $0$-minimal ideal $I$ with maximal subgroup $H=G_e$, where $e$ is a nonzero idempotent of $I$, and that:
\begin{enumerate}
\item the Cartan matrix of $\mathbb CM$ is symmetric;
\item $p\nmid |H|$;
\item $p$ divides the orders of the left and right stabilizers in $G$ of any element of $M\setminus G$;
\item there is a simple $kH$-module $T$ such that $\rad(\Ind_H(T))\cong S$ where $S$ is the trivial $kG$-module inflated to $kM$.
\end{enumerate}
Let $P$ be the projective cover of $S$.  Then
\begin{enumerate}
\item [(a)]  $Q=\Ind_H(T)$ is the projective cover of the simple module $S'=\Ind_H(T)/\rad(\Ind_H(T))$
\item [(b)]  $P$ contains a submodule $A\cong S\oplus Q$.
\item [(c)] Every endomorphism of $P$ preserves both summands of $A$.
\item [(d)] Every nilpotent endomorphism of $P$ annihilates $A$.
\end{enumerate}
In particular, $kM$ has infinite representation type.
\end{Lemma}
\begin{proof}
First note that the Cartan matrix of $kM$ is symmetric by Theorem~\ref{t:CartanSquare} and (1).  We shall exploit this fact for two particular simple modules.  We may assume that $P=kM\varepsilon$ with $\varepsilon$ a primitive idempotent of $kM$.

Denote by $z$ the zero of $M$.  Note that since $G$ is the stabilizer of $z$, $p\mid |G|$ by (3) and hence $I$ is a proper ideal of $M$, as $p\nmid H$.   Note that $z$ is a central idempotent and $kMz\cong k$, and so $kM\cong kM/kz\times k$ as rings.  Then $kM(e-z)\cong kL_e$  as a $kM$-module by $0$-minimality of $I$, and so $kL_e$ is a projective module. Note that $kI/kz$ is isomorphic as a left $kM$-module to a direct sum of $t$ copies of $kL_e$, where $t$ is the number of nonzero $\eL$-classes of $I$,  by~\cite[Lemma~5.18]{repbook}.  Since $kH$ is semisimple, we have that $kL_e =\Ind_{H}(kH)\cong \bigoplus_{i=1}^r\Ind_{H}(S_i)$ where $S_1,\ldots, S_r$ are the simple $kH$-modules with multiplicities.  Let us assume without loss of generality that $S_1=T$. Note that since $\Ind_H(T)$ is indecomposable (cf.~\cite[Corollary~4.10]{repbook})  and a direct summand in the projective module $kL_e$, we deduce that it is the projective cover $Q$ of the simple module $S'=\Ind_H(T)/\rad(\Ind_H(T))$.  Note that $eS'\cong T$, whereas $eS=0$, and so $S'\ncong S$.  Thus  $S$ is a composition factor of the projective cover $Q$ of $S'$ with multiplicity $1$, and so by symmetry of the Cartan matrix, $S'$ is a composition factor of $P$ with multiplicity $1$, that is, $\dim_k\Hom_{kM}(Q,P)=1$.

In particular,  there is nonzero homomorphism $\psi\colon Q\to P$, which is unique up to multiplication by a scalar in $k^\times$.  Since $\rad(Q)\cong S$ is simple, if $\psi$ is not injective, then it has image $S'$.  But then $S'$ would be in $\soc P\subseteq \soc {kM}$.  We claim that $S'$ is not isomorphic to a submodule $V$ of $\soc{kM}$.  Indeed, if it were, then since $T\cong eS'\neq 0$, we have that $V=KMeV\subseteq kI$.  But $kI\cong  kL_{e}^t\oplus k$ (where $k$ is the trivial $kM$-module).   Recalling that $kL_e\cong \oplus_{i=1}^r\Ind_{H}(S_i)$, it follows that $S'$ is not isomorphic to any submodule of $kI$ as $\Ind_H(S_i)$ does not have $S'$ as a composition factor if $S_i\ncong T$ by~\cite[Theorem~5.5]{repbook}, whereas $\Ind_H(T)$ has socle $S\ncong S'$.   We conclude that $\psi\colon Q\to P$ is injective.  Since $kIQ=kMeQ=Q$, it follows that $\psi(Q)\subseteq kI\psi(Q)\subseteq kI$.

We next show that $P$ contains a second submodule $U$ isomorphic to $S$ with $U\cap \psi(Q)=0$.   Let  $U$ be the $k$-span of $a=\sum_{g\in G} g$.  We claim that $a\in Z(kM)$ and spans an isomorphic copy $U$ of $S$.  Obviously, $ga=a=ag$ for all $g\in G$.  If $m\in M\setminus G$, then by assumption the orders of the left and right stabilizers of $m$ are divisible by $p$.  But $am=|G_m|\sum_{m'\in Gm}m'=0$ (where $G_m$ is the left stabilizer of $m$), and, similarly, $ma=0$.     This argument shows that $a\in Z(kM)$ and $U\cong S$. Note that $a\varepsilon =\varepsilon a=a$, as $kMa\cong kM\varepsilon/\rad(kM)\varepsilon=k(\varepsilon +\rad(kM)\varepsilon)$ implies that $\varepsilon a=a$.  Thus $U\subseteq P$.  Since $U\cap kI=0$, we deduce that $A=U\oplus \psi(Q)$ is a submodule of $P$ (necessarily contained in the radical as $P$ is indecomposable projective).  Also $U$ and $W=\soc{\psi(Q)}$ are isomorphic copies of $S$ intersecting trivially.

Since $\End_{kM}(P) \cong \varepsilon kM\varepsilon$ (acting on the right), it remains to show that $U\varepsilon kM\varepsilon\subseteq U$,  $\psi(Q)\varepsilon kM\varepsilon\subseteq \psi(Q)$ and $U\varepsilon \rad(kKM)\varepsilon=\psi(Q)\varepsilon \rad(kM)\varepsilon=0$.  We proceed to do this.  Let $b\in \varepsilon kM\varepsilon$.  Then $ab=ba\in U$.  On the other hand, $u\mapsto \psi(u)b$ is a homomorphism $Q\to P$, and hence is a multiple of $\psi$.  Thus $\psi(Q)b\subseteq \psi(Q)$. If $b\in \varepsilon \rad(kKM)\varepsilon$, then $ab=ba=0$ as $U$ is simple and $b\in \rad(kM)$.  On the other hand, every nonzero homomorphism from $Q\to P$ is a nonzero scalar multiple of $\psi$, and hence has the same image as $\psi$.   Since the endomorphism of $P$ given by right multiplication by $b$ is nilpotent, we cannot have $\psi(Q)b=\psi(Q)$.  Therefore, $q\mapsto \psi(q)b$ is the zero homomorphism $Q\to P$, and so $\psi(Q)b=0$.  The final statement follows by applying Dickson's theorem (Theorem~\ref{t:dickson}) to the indecomposable projective $P$ and the submodules $U\cong W$, noting that $\Ext^1_{kM}(P,S)=0$ because $P$ is projective.  This completes the proof.
\end{proof}

As a corollary, we obtain a new proof of Glover's theorem.

\begin{Thm}[Glover]\label{t:glover}
Let $q$ be a power of the prime $p$ and $k$ a field of characteristic $p$ containing $\mathbb F_q$.  Then $kM(n,q)$ has infinite representation type for $n\geq 2$.  On the other hand, $kM(1,q)\cong k^q$ is semisimple.
\end{Thm}
\begin{proof}
If $n=1$, then $M(1,q)$ is just the multiplicative monoid of $\mathbb F_q$. Since the zero of $\mathbb F_q$ is a central idempotent, it follows that $k\mathbb F_q\cong k\times k\mathbb F_q^\times\cong k^q$ as $\mathbb F_q^\times$ is cyclic of order $q-1$.

It is well known that $\mathbb F_q$ is a splitting field for $\mathrm{GL}(n,q)$.
One can easily deduce that if $q\neq p$ or $n>2$, then $k\mathrm{GL}(n,q)$ has infinite representation type (because the $p$-Sylow subgroups of $\mathrm{GL}(n,q)$ are not cyclic~\cite{Webbbook}), and so $n=2$ is the only interesting case.  But we give a direct reduction to the case $n=2$.  Let \[e=\begin{bmatrix} I_2 & 0_{2,n-2} \\ 0_{n-2,2} & 0_{n-2}\end{bmatrix}.\]  Then $eM(n,q)e\cong M(2,q)$.  Now if $A$ is any algebra and $e$ is an idempotent, then the functor $F(V)=Ae\otimes_{eAe}V$ from $eAe$-modules to $A$-modules preserves indecomposability~\cite[Proposition~4.10]{repbook} and reflects isomorphism (since $eF(V)\cong V$).  Hence, if $eAe$ has infinite representation type, then so does $A$.  It, therefore, suffices to prove that $kM(2,q)$ has infinite representation type.

We verify the hypotheses of Lemma~\ref{l:infinite.rep.type.small}.
 The algebra $\mathbb CM(2,q)$ is semisimple~\cite{putchasemisimple,Kovacs,repbook} (although the case $n=2$ is easier) and hence has a symmetric Cartan matrix.   Let $I$ be the ideal of all matrices of rank at most $1$.  Then $I$ is $0$-minimal and if $e=E_{11}$, then $H=G_e\cong \mathbb F_q^\times$, as it consists of all $2\times 2$ matrices with upper left corner in $\mathbb F_q^\times$ and all remaining entries $0$, and so $p\nmid |H|$.   Also, $L_e$ consists of the rank $1$ matrices whose second column is zero.

If $A$ is a  rank one matrix, the left stabilizer in $\mathrm{GL}(2,q)$ of $A$ is the pointwise stabilizer of the image of $A$ on column vectors.  Similarly, the right stabilizer is the pointwise stabilizer of the image of $A$ on row vectors.  Both these stabilizers have size $q(q-1)$, which is divisible by $p$.

Let $S$ be the trivial $k\mathrm{GL}(2,q)$-module, inflated to $kM(2,q)$. We show that $S$ is the socle of $\Ind_H(k)$ where $k$ is the trivial module. The right action of $H$ on an element of $L_e$ just multiplies the first column (which is the only nonzero column) by a nonzero scalar.   Thus, as $kM$-modules, $\Ind_H(k)\cong k[L_e/H]\cong k\mathbb P(\mathbb F_q^2)$ where $\mathbb P(\mathbb F_q^2)$ is the projective space of lines in $\mathbb F_q^2$.  The action of a matrix $A$ is as in \eqref{eq:line.act}.

 Then the $k$-span of the sum $y$ of all the lines is $M(2,q)$-invariant and provides a copy of $S$.  Indeed, it is clearly fixed by all invertible matrices.  If $A$ is a rank $1$-matrix, then it annihilates one line and sends the remaining $q$ lines to its image $\ell$.  We deduce that $A\odot y=q\ell=0$.  Thus $ky\cong S$ as a submodule of $k\mathbb P(\mathbb F_q^2)$.   The quotient $k\mathbb P(\mathbb F_p^2)/S$ is a simple $kM(2,q)$-module $S'$.  Indeed,  any coset in $k\mathbb P(\mathbb F_q^2)/S$ can be uniquely represented by a linear combination $x=\sum_{\ell\in \mathbb P(\mathbb F_q^2)}c_{\ell}\ell$ with $\sum c_\ell=0$.  If $x\neq 0$ and $\ell_0\in \mathbb P(\mathbb F_q^2)$, then  choose a line $\ell'$ with $c_{\ell'}\neq 0$ and choose a rank one matrix $A$ with kernel $\ell'$ and image $\ell_0$.  Then $A\odot x = \sum_{\ell\neq \ell'}c_{\ell}\ell_0 =-c_{\ell'}\ell_0$.  Since $\ell_0$ was arbitrary, we deduce that the submodule generated by $x+S$ is $S'$.  This completes the proof.
\end{proof}

For the case, of $kM(2,2)$ with $k$ a field of characteristic $2$, we can give a quiver presentation of the basic algebra.  It turns out to be a string algebra of infinite representation type, which was essentially studied by Gelfand and Ponomarev~\cite{lorentz}.

\begin{Thm}\label{t:m(2,2)}
Let $k$ be a field of characteristic $2$.  Then the basic algebra of $kM(2,2)$ is given by the quiver presentation with quiver
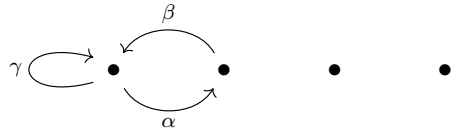
\[\begin{tikzcd}[cells={nodes={}}]\arrow[loop left, distance=3em, start anchor={[yshift=-1ex]west}, end anchor={[yshift=1ex]west}]{}{\gamma}\arrow[bend right=60,swap]{r}{\alpha}\bullet &\bullet\arrow[bend right=60,swap]{l}{\beta} &\bullet & \bullet\end{tikzcd}\] and relations $\alpha\beta=\gamma\beta=\gamma^2=\alpha\gamma=0$.  Thus $kM(2,2)$ is Morita equivalent to a string algebra of infinite representation type.
\end{Thm}
\begin{proof}
Note that $\mathrm{GL}(2,2)\cong S_3$ and $\mathbb F_2^\times$ is trivial.  Thus every field is a splitting field for $M(2,2)$ by Theorem~\ref{t:Brauer+CMP}, and so $(\mathbb Q_2,\mathbb Z_2,\mathbb F_2)$ is a splitting $2$-modular system. The general case will follow from the result for $\mathbb F_2$ by extending the scalars, and so we work over $\mathbb F_2$.

 The decomposition map takes both the trivial representation $T_1$ and the sign representation $T_2$ of $\mathbb Q_2S_3$ (inflated to $\mathbb Q_2M(2,2)$) to the trivial representation $T_1'$ of $\mathbb F_2S_3$ (inflated to $\mathbb F_2M(2,2)$).  The $2$-dimensional irreducible representation $T_3$ of $\mathbb Q_2S_3$ (inflated to $\mathbb Q_2M(2,2)$) is the set of vectors with coordinates summing to $0$ in $\mathbb Q_2^3$ with the action by permuting the coordinates.  This is sent by the decomposition map to the submodule $T_3'$ of $\mathbb F_2^3$ of vectors with coordinates summing to $0$, which is still simple as an $\mathbb F_2S_3$-module.  The decomposition map takes the trivial $\mathbb Q_2M(2,2)$-module $T_4$ to the trivial $\mathbb F_2M(2,2)$-module $T_4'$ and it takes the simple $\mathbb Q_2M(2,2)$-module $T_5=\mathbb Q_2\mathbb P(\mathbb F_2^2)$ to $[T_1']+[S']$ where  $S'$ is as in the proof of Theorem~\ref{t:glover} (see also the discussion in Example~\ref{ex:singular.2.2}). Thus
the decomposition matrix (with respect to the orderings indicated above) is
\[D= \begin{bmatrix} 1& 0 &0 &0 \\ 1&0 & 0 &0\\ 0 & 1 & 0 & 0\\ 0 & 0 & 1& 0\\ 1& 0&0&1\end{bmatrix}.\]
 Since $\mathbb Q_2M(2,2)$ is semisimple, we have by Theorem~\ref{t:CartanSquare} \[C(\mathbb FM(2,2)) = D^TD=\begin{bmatrix}3&0&0&1\\ 0 &1&0&0\\ 0&0&1&0\\ 1 &0 & 0 &1\end{bmatrix}\]
Therefore, $[T_3']$ and $[T_4']$ are isolated vertices in the quiver, there are no loops at $[S']$ and there are single edges $\beta\colon [S']\to [T_1']$ and $\alpha\colon [T_1']\to [S']$.  Lemma~\ref{l:infinite.rep.type.small} and the proof of Theorem~\ref{t:glover} show that the projective cover $Q=\mathbb F_2\mathbb P(\mathbb F_2^2)$ of $S'$ has radical isomorphic to $T_1'$ and hence $\rad^2(Q)=0$. Therefore, $\alpha\beta=0$.  However, $\beta\alpha\neq 0$ since we saw in the proof of Lemma~\ref{l:infinite.rep.type.small} that $S'$ is not in the socle of the projective cover $P$ of $T_1'$.   But then the empty path and $\beta\alpha$ provide at most $2$ composition factors of $P$ isomorphic to $T_1'$, and so there must be a loop $\gamma$ at $[T']$.  Then $\gamma\beta=0$ because $\rad^2(Q)=0$.    Finally, we must have $\gamma^2=0=\alpha\gamma$.  Indeed, Lemma~\ref{l:infinite.rep.type.small} yields that $\rad(P)$ contains a submodule isomorphic to $T_1'\oplus Q$ and any element of $\varepsilon \rad(\mathbb F_2M(2,2))\varepsilon$ annihilates this submodule under right multiplication (where $P=\mathbb F_2M(2,2)\varepsilon$ with $\varepsilon$ a primitive idempotent).  Looking at the Cartan matrix, we see that $\rad(P)\cong T_1'\oplus Q$, and so $\gamma^2=0=\alpha\gamma$.

It follows that the basic algebra is given by the above quiver presentation, and it is a string algebra~\cite{stringalgebra}.  Note that the strings $(\alpha\gamma\inv\beta)^n$ with $n\geq 1$ give distinct string modules, providing another proof that this algebra has infinite representation type.
\end{proof}

We now use Lemma~\ref{l:infinite.rep.type.small} to construct, for each prime $p>0$, a monoid $M$ such that each maximal subgroup of $M$ has finite representation type over every field, $KM$ is semisimple whenever $K$ has characteristic different than $p$ and has infinite representation type when $K$ has characteristic $p$.   Define $M_p = C_p\cup T$ where $C_p$ is a cyclic group of order $p$ generated by $x$ and $T=\{z\}\cup  \{0,\ldots, p\}^2$.  The element $z$ is a zero element and \[(i,j)(k,\ell) = \begin{cases}(i,\ell), & j\neq k\\ z, & \text{else.}\end{cases}\]  We have $xs=s=sx$ for all $s\in T$.  Then $M_p$ has a cyclic group $C_p$ of order $p$ for the group of units and all other maximal subgroups are trivial, $T$ is a $0$-minimal ideal and $T$ is the Rees matrix semigroup over the trivial group with sandwich matrix $J-I$ where $J$ is the $(p+1)\times (p+1)$ all ones matrix and $I$ is the identity matrix.

\begin{Prop}\label{p:all.but.p}
The algebra $KM_p$ is semisimple if the characteristic of $K$ is different than $p$.  If $K$ has characteristic $p$, then $KM_p$ has infinite representation type, but each maximal subgroup of $M$ has finite representation type over $K$.
\end{Prop}
\begin{proof}
By construction, the sandwich matrix for the idempotent $e=(0,1)$ is $J-I$.   The matrix $J$ has eigenvalues $0$ and $p+1$, and the latter eigenvalue is a simple eigenvalue with eigenvector the all ones vector if the characteristic of $K$ does not divide $p+1$.   Thus $1$ is an eigenvalue of $J$ over $K$ if and only if $p+1$ is congruent to $1$ modulo the characteristic of $K$, in which case the all ones vector spans the null space of $J-I$.  Since all other sandwich matrices for $M_p$ are $1\times 1$ identity matrices, we conclude that $KM_p$ is semisimple for $K$ of every characteristic except $p$.  If $K$ has characteristic $p$, then $C_p$ has just the trivial representation over $p$, and so $K$ is a splitting field by Theorem~\ref{t:Brauer+CMP}.  Also $KC_p\cong K[x]/(x^p-1)=K[x]/(x-1)^p$ and hence has finite representation type (the indecomposables being given by Jordan forms $J_n(1)$ with $1\leq n\leq p$).  We now verify that $KM_p$ has infinite representation type using Lemma~\ref{l:infinite.rep.type.small}.  Since $\mathbb CM_p$ is semisimple, it has a symmetric Cartan matrix.  The maximal subgroup of $T\setminus \{z\}$ is trivial. The left/right stabilizer of each element of $T$ is $C_p$, and so has order divisible by $p$.  Finally, $KL_e$ has basis $\{(i,1)\mid 0\leq i\leq  p\}$ and the null space of $J-I$ (which is $\rad(KL_e)$) is spanned by the all ones vector, which is the coordinate vector for $v=\sum_{\ell=0}^p(\ell,1)$.  Clearly, $C_p$ fixes $v$ and $z$ annihilates it.  Also, $(i,j)\sum_{\ell =0}^p(\ell,1) = p(i,1)=0$.  Therefore, Lemma~\ref{l:infinite.rep.type.small} applies to deduce  $KM_p$ has infinite representation type.
\end{proof}

It is not too difficult to show that if $K$ has characteristic $p$, then the basic algebra of $KM_p$ is a string algebra of infinite representation type with quiver
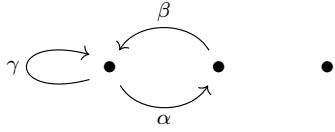
\[\begin{tikzcd}[cells={nodes={}}]\arrow[loop left, distance=3em, start anchor={[yshift=-1ex]west}, end anchor={[yshift=1ex]west}]{}{\gamma}\arrow[bend right=60,swap]{r}{\alpha}\bullet &\bullet\arrow[bend right=60,swap]{l}{\beta} &\bullet\end{tikzcd}\] and relations $\alpha\beta=\gamma\beta=\alpha\gamma=\gamma^p=0$.  The argument is similar to that of Theorem~\ref{t:m(2,2)}.

\section{Modular versus nonmodular characteristic}
The purpose of this section is to define modular and strongly modular characteristics, prove the decomposition matrix is a permutation matrix precisely in the non-strongly modular case, and show that a quiver presentation of $kM$ can be lifted to characteristic zero in the nonmodular case, making precise in which sense the characteristic zero theory agrees with the nonmodular positive characteristic setting. All but finitely many primes will be nonmodular.  Of course, everything reduces to the usual definitions for groups.
Here we will need to use a bit more  monoid representation theory, as well as some notions about quivers.

Let $M$ be a finite monoid and $e\in E(M)$ an idempotent.   If $A$ is a commutative ring, then $AL_e$ is a free right $AG_e$-module and $AR_e$ is a free left $AG_e$-module via multiplication.  In fact, $AL_e$ is an $AM$-$AG_e$-bimodule (with left action via \eqref{eq:schutz})  and, dually, $AR_e$ is an $AG_e$-$AM$-bimodule (see~\cite[Chapter~5]{repbook} where this is done for the case $A$ is a field).  If $V$ is an $AG_e$-module, then the \emph{induced} $AM$-module is  $\Ind_{G_e}(V) = AL_e\otimes_{AG_e} V$ and the \emph{coinduced} $AM$-module is $\Coind_{G_e}(V)=\Hom_{AG_e}(AR_e,V)$.  Let us put $AR_e^{\vee}=\Coind_{G_e}(AG_e)=\Hom_{AG_e}(AR_e,AG_e)$.  Note that $AR_e^\vee$ is an $AM$-$AG_e$-bimodule which is free as a right $AG_e$-module because $AR_e$ is a free left $AG_e$-module on $G\backslash R_e$, and hence $AR_e^{\vee}\cong (AG_e)^{|G_e\backslash R_e|}$ as a right $AG_e$-module. The functors $\Ind_{G_e}$ and $\Coind_{G_e}$ are exact because $AL_e$ and $AR_e$ are free $AG_e$-modules on the appropriate side.

There is a natural bimodule homomorphism $\p_e\colon \mathbb ZL_e\to \mathbb ZR_e^{\vee}$ given on $\ell\in L_e$ and $r\in R_e$ by
\[\p_e(\ell)(r) = \begin{cases}r\ell, & \text{if}\ r\ell\in G_e\\ 0 , & \text{else}\end{cases}\] where we note that $r\ell\in eMe$ and belongs to $G_e$ if and only if $r\ell\in J_e$ (cf.~\cite[Chapter~1]{repbook}).  That this is a bimodule homomorphism is verified in~\cite[Chapter~5]{repbook} over fields, but the proof is general.   If we choose as a $\mathbb ZG_e$-basis for $\mathbb ZL_e$ orbit representatives for $L_e/G_e$ and as a $\mathbb ZG_e$-basis for $\mathbb ZR_e^{\vee}$ the dual basis to a $\mathbb ZG_e$-basis of $\mathbb ZR_e$ consisting of orbits representatives for $G_e\backslash R_e$, then the matrix of $\p_e$ as a homomorphism of free right $\mathbb ZG_e$-modules takes on values in $G_e\cup \{0\}$ and is known in semigroup theory as the sandwich matrix of $J_e$ (see~\cite[Chapter~5]{repbook} for details).  Notice that $\mathbb ZL_e$ and $\mathbb ZR_e^{\vee}$ are finitely generated free abelian groups as they are finitely generated free right $\mathbb ZG_e$-modules and $G_e$ is finite.

For a number of reasons, the following alternate description of $\Coind_{G_e}$ will be convenient.  It is a special case of the fact that if $R$ is a ring and $P$ is a finitely generated projective $R$-module, then the natural map $\Hom_R(P,R)\otimes_R V\to \Hom_R(P,V)$ is an isomorphism for any finitely presented $R$-module $V$.

\begin{Prop}\label{p:natural.iso.coind}
Let $M$ be a finite monoid, $e\in E(M)$ and $A$ a commutative ring.  Then there is an isomorphism $\eta_V\colon AR_e^{\vee}\otimes_{AG_e}V\to \Coind_{G_e}(V)$ of $AM$-modules, natural in $V$, for each finitely presented $AG_e$-module $V$ given by $\eta_V(f\otimes v)(r) = f(r)v$.  If $A$ is Noetherian, then finitely presented can be replaced by finitely generated.
\end{Prop}

If $A$ is a commutative ring with unit, then $A\otimes_{\mathbb  Z}\mathbb ZL_e\cong AL_e$ and $A\otimes_{\mathbb Z}\mathbb ZR_e^{\vee}\cong AR_e^\vee$.  Only the latter claim requires any verification.  Indeed, fix a transversal $T$ to $G_e\backslash R_e$.   Since $\mathbb ZR_e$ is a free left $\mathbb ZG_e$-module on $T$, we can identify $\mathbb ZR_e^{\vee}$ with $\mathbb ZG_e^T$.  Similarly we can identify $AR_e^{\vee}$ with $AG_e^T$.  Since $T$ is finite, we have right $AG_e$-module isomorphisms  $A\otimes_{\mathbb Z} \mathbb ZR_e^{\vee}\cong A\otimes_{\mathbb Z}\mathbb ZG_e^T\cong AG_e^T\cong  AR_e^{\vee}$ and the isomorphisms are also $AM$-module isomorphisms.

Let us denote $1_A\otimes \p_e\colon AL_e\to AR_e^{\vee}$ by $\p_{e,A}$ for $e\in E(M)$.   We put
$T_e=\p_e(\mathbb ZL_e)\subseteq \mathbb ZR_e^{\vee}$, which is a finitely generated free abelian group.  To decongest notation,
let $T_e(A)=A\otimes_{\mathbb Z} T_e$ for a commutative ring $A$.  Note that if $A$ is a subring of $B$, then
$T_e(B)\cong B\otimes_A T_e(A)$ by transitivity of extension of scalars.  Of course, $T_e(A)$ is a finitely generated free $A$-module.

If $K$ is a field and $V$ is a simple $KG_e$-module, then $\Ind_e(V)$ has simple top, $\Coind_e(V)$ has simple socle and under the identification of Proposition~\ref{p:natural.iso.coind}, $\rad(\Ind_e(V))=\ker (\p_{e,K}\otimes 1_V)$ and the socle of $\Coind_e(V)$ is the image of $\p_{e,K}\otimes 1_V$.  See~\cite[Chapter~5]{repbook} for details.  Thus the image of $\p_{e,K}\otimes 1_V$ is the simple $KM$-module associated to $V$ by Clifford-Munn-Ponizovskii theory.

We shall also need  the following straightforward observation.  Suppose that $A$ is a subring of $B$.  Then extension of scalars commutes with induction and coinduction, that is, if $V$ is a finitely presented $AG_e$-module, then $B\otimes_A\Ind_{G_e}(V)\cong \Ind_{G_e}(B\otimes_A V)$ and $B\otimes_A\Coind_{G_e}(V)\cong \Coind_{G_e}(B\otimes_A V)$ (natural in $V$). The first isomorphism follows directly from Lemma~\ref{l:commute.with.ext}.  For finitely presented modules $\Coind_{G_e}$ commutes with extension of scalars by Proposition~\ref{p:natural.iso.coind} and Lemma~\ref{l:commute.with.ext}.

 If $M$ is a finite monoid, we say that a prime $p>0$ is \emph{strongly modular} for $M$ if there is an idempotent $e$ such that $p$ divides either the order of $G_e$ or the order of the torsion subgroup of the finitely generated abelian group $\coker \p_e$ (this depends only on the $\mathscr J$-class of $e$).    For example, the prime $2$ is strongly modular for the monoid in Example~\ref{ex:singular.2.2} because the sandwich matrix \eqref{eq:sandwich.examp} for the $\J$-class of rank one maps has determinant $2$ and hence the cokernel is a finite elementary abelian $2$-group.  All other primes are non-strongly modular for this example.

Let $M$ be a finite monoid and $p>0$ a prime which is non-strongly modular for $M$.  Let $(F,\mathcal O,k)$ be a splitting $p$-modular system.  Our goal is to show the following facts:  each simple $kM$-module $S$ can be lifted to an $\mathcal OM$-lattice $\til S$ such that $[S]\mapsto [F\otimes_{\mathcal O} \til S]$ gives a bijection between simple $kM$ and simple $FM$-modules (preserving apexes and dimensions); the decomposition matrix is a permutation matrix conjugating the Cartan matrices of $FM$ and $kM$; and the projective cover of $F\otimes_{\mathcal O}\til S$ is $F\otimes_{\mathcal O} P$ where $P$ is the projective cover of $S$ as an $\mathcal OM$-module.

\begin{Lemma}\label{l:kernel.injects}
Let $M$ be a finite monoid and $e\in E(M)$.  Let $p$ be a prime (possibly $0$) and $K$ a field of characteristic $p$.  Then the image of $\p_{e,K}$ is a quotient of $T_e(K)$ and is isomorphic $T_e(K)$ if and only if $p$ does not divide the order of the torsion subgroup of $\coker \p_e$.
\end{Lemma}
\begin{proof}
 We shall use throughout that the image of a homomorphism $f\colon A\to B$ of modules can be identified with the kernel of the projection $B\to \coker f$.  In particular, exact functors preserves images, and so without loss of generality we may assume that $K$ is the prime field as any field extension of a field is flat and extending scalars to a larger field reflects isomorphisms.

Since tensor products are right exact, they preserve cokernels. Thus $\coker \p_{e,K}=K\otimes_{\mathbb Z} \coker \p_e$.  We derive from the exact sequence of abelian groups $0\to T_e\to \mathbb ZR_e^{\vee}\xrightarrow{\pi} \coker \p_e\to 0$ an exact sequence \[0\to \Tor^{\mathbb Z}_1(K,\coker \p_e)\to T_e(K)\to KR_e^{\vee}\xrightarrow{\pi} \coker \p_{e,K}\to 0.\]   Since $K$ is the prime field, if $p=0$, then $K=\mathbb Q$ is flat as a $\mathbb Z$-module, and so $\Tor^{\mathbb Z}_1(K,\coker \p_e)=0$, whence $T_e(K)$ is isomorphic to the image of $\p_{e,K}$.  If $p>0$, then $K\cong \mathbb Z/p\mathbb Z$ as an abelian group, and so $\Tor^{\mathbb Z}_1(K,\coker \p_e)=\{a\in \coker\p_e\mid pa=0\}$.  Thus the surjective map $T_e(K)\to \ker \pi$ is injective if and only if $p$ does not divide the order of the torsion subgroup of $\coker \p_e$.  Moreover, if $p$ divides this order, then $\dim_K T_e(K)>\dim_K \ker \pi$.  Since $\ker \pi$ is the image of $\p_{e,K}$,  this completes the proof.
\end{proof}

\begin{Cor}\label{c:simple.field.ext}
Let $M$ be a finite monoid and $K$ a field of characteristic $p\geq 0$.  Suppose that $e\in E(M)$ such that $p\nmid |G_e|$ and  $V$ is a simple $KG_e$-module.  Then the simple $KM$-module associated to $V$ is a quotient of $T_e(K)\otimes_{KG_e} V$.  Moreover, they are isomorphic if  $p$ does not divide the order of the torsion subgroup of $\coker \p_e$.
\end{Cor}
\begin{proof}
Lemma~\ref{l:kernel.injects} shows that the image of $\p_{e,K}$ is a quotient of $T_e(K)$ and they are isomorphic if and only if $p$ does not divide the order of the torsion subgroup of $\coker \p_e$.  Since $KG_e$ is semisimple, $V$ is projective, hence flat, and so tensoring with $V$ preserves images.  Thus $T_e(K)\otimes_{KG_e} V$ maps onto the image of $\p_{e,K}\otimes 1_V$ and they are isomorphic if  $p$ does not divide the order of the torsion subgroup of $\coker \p_e$.
\end{proof}

Let us recall (cf.~\cite[Theorem~9.4.11]{Webbbook} and its proof) that if $G$ is a finite group, $p>0$ is a prime not dividing $|G|$ and $(F,\mathcal O,k)$ is a splitting $p$-modular system, then, for each simple $kG$-module $S$, there is a projective indecomposable $\mathcal OG$-module $\til S$ with $k\otimes_{\mathcal O} \til S\cong S$ such that $F\otimes_{\mathcal O} \til S$ is simple and the correspondence $[S]\mapsto [F\otimes_{\mathcal O}\til S]$ is a bijection between isomorphism classes of simple $kG$-modules and simple $FG$-modules that is inverse to the decomposition map.  Hence the decomposition matrix is  a permutation matrix.

We now prove our first theorem about the non-strongly modular case.

\begin{Thm}\label{t:char.nonmodular}
Let $M$ be a finite monoid and $p>0$ a prime.  Let $(F,\mathcal O,k)$ be a splitting $p$-modular system for $M$.
\begin{enumerate}
  \item  Suppose that $p$ does not divide the order of any maximal subgroup of $M$. Then there is a bijection between the simple $kM$-modules and the simple $FM$-modules, $V\mapsto \ov V$, such that $\dim_k V\leq \dim_F \ov  V$.  Moreover, the decomposition matrix is unipotent upper triangular with respect to an appropriate ordering of the simple modules.
  \item If $p$ is non-strongly modular for $M$, then, for each simple $kM$-module $V$, there is an $\mathcal OM$-lattice $\til V$  such that $k\otimes_{\mathcal O}\til V\cong V$ and $F\otimes_{\mathcal O} \til V=\ov  V$ is simple.  Moreover, the decomposition matrix is a permutation matrix and the Cartan matrix $C(FM)$ is conjugate to the Cartan matrix $C(kM)$ by a permutation matrix. If $P$ is a projective indecomposable $\mathcal OM$-module with simple top $V$, then $F\otimes_{\mathcal O} P$ is a projective indecomposable $FM$-module with simple top $F\otimes_{\mathcal O} \til V$.
  \item If $p$ is strongly modular for $M$, then the decomposition matrix is not a permutation matrix.
\end{enumerate}
\end{Thm}
\begin{proof}
The assumptions of (1) imply  that $FM$ and $kM$ both have the same number of isomorphism classes of simple modules, namely the number of generalized conjugacy classes of $M$.
Let $e\in E(M)$. Note that $(F,\mathcal O,k)$ is a splitting $p$-modular system for $G_e$ by Theorem~\ref{t:Brauer+CMP}.
 Suppose now that $S$ is a simple $kG_e$-module.
Since $p$ does not divide $|G_e|$, by the above discussion, there is a projective indecomposable $\mathcal OG_{e}$-module $\til S$ with $F\otimes_{\mathcal O} \til S$ simple and $k\otimes_{\mathcal O} \til S\cong S$.
Then $T_e(F)\otimes_{FG_e} (F\otimes_{\mathcal O} \til S)$ is the simple module associated to $F\otimes_{\mathcal O}\til S$ by Corollary~\ref{c:simple.field.ext}, since $F$ has characteristic $0$. If $V_S$ is the simple $kM$-module corresponding to $S$, we put $\ov V_S=T_e(F)\otimes_{FG_e} (F\otimes_{\mathcal O} \til S)$.   Note that since $\til S$ is a direct summand in $\mathcal OG_e$, we have that $\til V_S=T_e(\mathcal O)\otimes_{\mathcal OG_e} \til S$ is a direct summand in $T_e(\mathcal O)\cong T_e(\mathcal O)\otimes_{\mathcal OG_e}\mathcal OG_e$, and hence an $\mathcal OM$-lattice.
Moreover, $\ov V_S\cong F\otimes_{\mathcal O} \til V_S$ and $k\otimes_{\mathcal O} \til V_S\cong  T_e(k)\otimes_{kG_e} (k\otimes_{\mathcal O}\til S)\cong T_e(k)\otimes_{kG_e} S$ by Lemma~\ref{l:commute.with.ext}.
By Corollary~\ref{c:simple.field.ext}, we have that $V_S$ is a quotient of $T_e(k)\otimes_{kG_e} S$ and so $\dim_F \ov V_S\geq \dim_k V_S$.  Moreover, $d([\ov V_S]) = [V_S]+[W]$ where $W$ is the kernel of the projection $T_e(k)\otimes_{kG_e} S\to V_S$.   Note that $T_e(k)\otimes_{kG_e} S$ is a quotient of $\Ind_{G_e}(S)=kL_e\otimes_{kG_e} S$ since the tensor product is right exact.  It is known~\cite[Theorem~5.5]{repbook} that all composition factors of $\Ind_{G_e}(S)$ other than $V_S$ have apex $f$ with $MeM<MfM$.  Thus $[W]$ is a sum of simple modules with apexes generating strictly larger principal ideals than $MeM$.
Since $FM$ and $kM$ have the same number of simple modules, we see that $V_S\mapsto \ov V_S$ is a bijection between the simple modules, and $D^T$ is unipotent lower triangular with respect to an ordering of the simple modules compatible with the natural ordering on their apexes (improving on Corollary~\ref{c:good.char}, which just had $D^T$ unimodular).  This proves (1).

Assume now that $p$ is non-strongly modular (and hence divides the order of no maximal subgroup) and let $S$ be a simple $kG_e$-module. Then Corollary~\ref{c:simple.field.ext} shows that $V_S\cong T_e(k)\otimes_{kG_e} S$, retaining the above notation.   Thus $\til V_S=T_e(\mathcal O)\otimes_{\mathcal OG_e} \til S$ is an $\mathcal OM$-lattice lifting the simple module $V_S$, and $\ov V_S=F\otimes_{\mathcal O}\til V_S$.  Therefore, $d[\ov V_S]=[V_S]$ and $\dim_F \ov V_S=\dim_k V_S$.  We conclude that the decomposition matrix is a permutation matrix as the decomposition map sends the basis of simples for $G_0(FM)$ bijectively to the basis of simples for $G_0(kM)$.  Moreover, its transpose $e$ must send the dual basis for $K_0(kM)=K_0(\mathcal OM)$ of projective indecomposable $\mathcal OM$-modules to the dual basis for $K_0(FM)$ of projective indecomposable $FM$-modules.  Therefore, if $P$ is a projective indecomposable $\mathcal OM$-module with simple top $V_S$, then  $F\otimes_{\mathcal O} P$ is a projective  indecomposable $FM$-module.  Note that since $\til V_S/\mathfrak m\til V_S\cong V_S$ and $P$ is projective, we can lift the quotient map $P\to V_S$ to a homomorphism $\psi\colon P\to \til V_S$ with $\psi(P)+\mathfrak m\til V_S=\til V_S$.  Hence $\psi$ is surjective by Nakayama's lemma.  Therefore, we have a surjective homomorphism $F\otimes \psi\colon F\otimes_{\mathcal O}P\to F\otimes_{\mathcal O}\til V_S$, whence the projective indecomposable module $F\otimes_{\mathcal O} P$ has top  $\ov V_S$.  This establishes (2).

Suppose now that $p$ is strongly modular.  If $p\mid |G_e|$ for some idempotent $e$, then $kM$ has strictly fewer simple modules than $FM$ because not all generalized conjugacy classes are $p$-regular.  Thus the decomposition matrix is not a permutation matrix.  So suppose that $p$ divides the order of no maximal subgroup of $M$, but $p$ divides the order of the torsion subgroup of $\coker \p_e$.  Then $T_e(k)$ is not isomorphic to the image $U$ of $\p_{e,k}$ by Lemma~\ref{l:kernel.injects}.  Since $kG_e$ is semisimple, we have that $kG_e=S_1\oplus\cdots\oplus S_r$ with $S_1,\ldots, S_r$ simple.  Since $S_1,\ldots, S_r$ are projective, the image of $\p_{e,k}\otimes 1_{S_i}$ is $U\otimes_{kG_e} S_i$.  If $U\otimes_{kG_e}S_i\cong T_e(k)\otimes_{kG_e} S_i$ for  all $i=1,\ldots, r$, then we would get $U\cong U\otimes_{KG_e}\bigoplus_{i=1}^rS_i\cong \bigoplus_{i=1}^r (U\otimes_{kG_e} S_i)\cong \bigoplus_{i=1}^r (T_e(k)\otimes_{kG_e} S_i)\cong T_e(k)\otimes_{kG_e} \bigoplus_{i=1}^rS_i\cong T_e(k)$, a contradiction.  So we can find $i$ with $U\otimes_{kG_e} S_i\ncong T_e(k)\otimes_{kG_e} S_i$.  By Corollary~\ref{c:simple.field.ext}, we have that $U\otimes_{kG_e} S_i$ is a proper quotient of  $T_e(k)\otimes_{kG_e} S_i$.  Thus (retaining the notation of the first paragraph of the proof), we have $d([\ov V_{S_i}]) =[V_{S_i}]+[W]$ with $W\neq 0$, whence the decomposition matrix is not a permutation matrix.   This completes the proof.
\end{proof}

 Let $J=\rad(\mathbb QM)\cap \mathbb ZM$.  Note that $J$ is a nilpotent ideal of $\mathbb ZM$ and $\mathbb ZM/J$ is a ring embedding in $\mathbb QM/\rad(\mathbb QM)$, and hence  its additive group is finitely generated free abelian.   Also note that $\rad(\mathbb QM)=\mathbb Q\otimes_{\mathbb Z} J$ since, for any $a\in \rad(\mathbb QM)$, there is $n>0$ with $na\in J$.  Therefore, if $K$ is a field of characteristic zero, then $\rad(KM) = K\otimes_{\mathbb Z} J$ by transitivity of extension of scalars and Proposition~\ref{p:radical.extension}.

One can compute a basis for $J$ in the following way.  If $m\in M$, let $\mathrm{Fix}_L(m) =\{m_0\in M\mid mm_0=m_0\}$.  Notice that if $\theta$ is the character of the regular $\mathbb QM$-module $\mathbb QM$, then $\theta(m) =|\mathrm{Fix}_L(m)|$.    Let $M=\{m_1,\ldots, m_n\}$ and define a non-negative integer $n\times n$-matrix $L$ by $L_{ij}=|\mathrm{Fix}_L(m_im_j)|$.  Then $L$ is symmetric and is the matrix of the trace form $(a,b)\mapsto \theta(ab)$ with respect to the basis $M$.  Since we are in characteristic zero, $\rad(\mathbb QM)$ is the radical of trace form and hence can be identified with $\ker L$.  See~\cite{Drazin} for details, where this was first observed.  Since $L$ is an integer matrix, we can compute a basis of $J=\ker L\cap \mathbb ZM$ via standard methods.

\begin{Prop}\label{p:modular.rad}
Let $M$ be a monoid and $J=\rad(\mathbb QM)\cap \mathbb ZM$.  If $p$ is a non-strongly modular prime, then for any field $K$ of characteristic $p$, one has that the natural map $K\otimes_{\mathbb Z} J\to KM$ is injective with image $\rad(KM)$, that is, one can identify $\rad(KM)$ with $K\otimes_{\mathbb Z} J$. Moreover, $KM/\rad(KM)\cong K\otimes_{\mathbb Z}\mathbb ZM/J$ as rings and $\rad(KM)/\rad^2(KM)\cong K\otimes_{\mathbb Z} J/J^2$ as $KM$-bimodules.
\end{Prop}
\begin{proof}
The first statement was already observed for the case that $K$ is of characteristic $0$.  So assume that $p>0$ is non-strongly modular.   Let $(F,\mathcal O,k)$ be a splitting $p$-modular system for $M$.  Then since $F\otimes_{\mathbb Z}J=\rad(FM)$, it follows that $FM/\rad(FM)\cong F\otimes_{\mathbb Z} \mathbb ZM/J$.  Since $\mathbb ZM/J$ is free abelian, $\Tor^{\mathbb Z}_1(\mathbb F_p,\mathbb ZM/J)=0$, and so $0\to \mathbb F_p\otimes_{\mathbb Z} J\to \mathbb F_pM\to \mathbb F_p\otimes_{\mathbb Z}\mathbb ZM/J\to 0$ is an exact sequence of $\mathbb F_p$-vector spaces.  Note that the image of $\mathbb F_p\otimes_{\mathbb Z} J$ is a nilpotent ideal and hence contained in $\rad(\mathbb F_pM)$.

   Tensoring with $k$, and using transitivity of extension of scalars, we then obtain an exact sequence $0\to k\otimes_{\mathbb Z} J\to kM\to k\otimes_{\mathbb Z} \mathbb ZM/J\to 0$.  The image of $k\otimes_{\mathbb Z}J$ in $kM$ is a nilpotent ideal, and hence contained in the radical.   Note that $\dim_k kM/\rad(kM)=\dim_F FM/\rad(FM)$ by Theorem~\ref{t:char.nonmodular} since both are split semisimple algebras having the same number of simples with the same dimensions.  Thus $\dim_k \rad(kM) =|M|-\dim_k kM/\rad(kM) = |M|- \dim_F FM/\rad(FM) = \dim_F \rad(FM)= \dim_F F\otimes_{\mathbb Z}J = \dim_k k\otimes_{\mathbb Z} J$.  It follows that the image of $k\otimes_{\mathbb Z}J$ in $kM$ under the natural map is $\rad(kM)$.  It then follows from Proposition~\ref{p:radical.extension} that the image of $\mathbb F_p\otimes_{\mathbb Z}J$ in $\mathbb F_pM$ is $\rad(\mathbb F_pM)$ by dimension considerations, and hence the same is true for every field $K$ of characteristic $p$ by another application of Proposition~\ref{p:radical.extension}.  Since tensoring with $K$ over $\mathbb F_p$ is exact, it also follows that $KM/\rad(KM)\cong K\otimes_{\mathbb Z} \mathbb ZM/J$.

    Finally, note $\rad(KM)/\rad^2(KM)\cong (KM/\rad(KM))\otimes_{KM} \rad(KM)\cong K\otimes_{\mathbb Z} (\mathbb ZM/J\otimes_{\mathbb ZM} J)\cong K\otimes_{\mathbb Z}J/J^2$ by Lemma~\ref{l:commute.with.ext}.
\end{proof}

Recall that the \emph{Loewy length} of a finite dimensional algebra $A$ is the least $n$ such that $\rad(A)^n=0$.  Let $n$ be the Loewy length of $\mathbb QM$ (which equals the Loewy length of $\mathbb CM$ by Proposition~\ref{p:radical.extension}).
Let us define a prime $p$ to be \emph{modular} for $M$ if either it is strongly modular, or it divides the order of the torsion subgroup of $J^k/J^{k+1}$ for some $1\leq k< n$ (where $J=\rad(\mathbb QM)\cap \mathbb ZM$).  Note that there are only finitely many primes that are modular for $M$.    I do not know any example of a finite monoid with a modular prime that is not strongly modular.  So far, all monoids that I have been able to construct for which $J/J^2$ has $p$-torsion also have a maximal subgroup of order divisible by $p$.

We show that when $p$ is nonmodular for $M$, there is a quiver $Q$ and an ideal $I$ of $\mathcal OQ$ such that $F\otimes_{\mathcal O}I$ is an admissible ideal with $FQ/(F\otimes_{\mathcal O} I)$ isomorphic to the basic algebra of $FM$ and  $k\otimes_{\mathcal O} I$ is an admissible ideal of $kQ$ with $kQ/(k\otimes_{\mathcal O} I)$ isomorphic to the basic algebra of $kM$.  This shows that, in a sense, the characteristic $p$ representation theory is ``the same'' as the characteristic $0$ representation theory when $p$ is nonmodular.  First we show that the Loewy length and series looks the same in characteristic $0$ and nonmodular characteristic $p$.

\begin{Prop}\label{p:loewy.preserved}
Let $M$ be a finite monoid and $p\geq 0$ a nonmodular prime for $M$.  Let $n$ be the Loewy length of $\mathbb QM$ and let $J=\rad(\mathbb QM)\cap \mathbb ZM$.  Let $K$ be a field of characteristic $p$.  Then $KM$ has Loewy length $n$,  the natural map $K\otimes_{\mathbb Z}J^k\to KM$ is injective with image $\rad^k(KM)$ for all $1\leq k<n$ and $\rad^k(KM)/\rad^{k+1}(KM)\cong K\otimes_{\mathbb Z}J^k/J^{k+1}$ for all $0\leq k\leq n-1$.
\end{Prop}
\begin{proof}
We prove the result by induction on $k$.
Proposition~\ref{p:modular.rad} establishes the result for $k=1$. In particular, $\rad(KM)=K\otimes_{\mathbb Z}J$, and so the Loewy length of $KM$ is at most $n$.   Assume the result is true for $k<n-1$.   Let $F\subseteq K$ be the prime field.   First note that since $\Tor_1^{\mathbb Z}(F,J^{k}/J^{k+1})=0$ by definition of a nonmodular prime, tensoring  the exact sequence $0\to J^{k+1}\to J^{k}\to J^{k}/J^{k+1}\to 0$ with $F$ leads to an exact sequence $0\to F\otimes_{\mathbb Z}J^{k+1}\to F\otimes_{\mathbb Z}J^k\to F\otimes_{\mathbb Z}J^k/J^{k+1}\to 0$, which we can then tensor over $F$ with $K$ to obtain an exact sequence $0\to K\otimes_{\mathbb Z}J^{k+1}\to K\otimes_{\mathbb Z}J^k\to \mathbb K\otimes_{\mathbb Z}J^k/J^{k+1}\to 0$.  By induction the natural homomorphism $K\otimes_{\mathbb Z} J^k\to KM$ is injective with image $\rad^k(KM)$.  Therefore, the composition $K\otimes_{\mathbb Z}J^{k+1}\to K\otimes_{\mathbb Z} J^k\to KM$ is injective, and it is the natural map.  The image is the $K$-span of $J^{k+1}=JJ^k$.  From the inductive hypothesis this is $\rad(KM)\rad^k(KM) =\rad^{k+1}(KM)$.   Moreover, $K\otimes_{\mathbb Z}J^{k}/J^{k+1}\cong (K\otimes_{\mathbb Z}J^k)/(K\otimes_{\mathbb Z}J^{k+1})\cong  \rad^{k}(KM)/\rad^{k+1}(KM)$.  This completes the induction.  In particular, it follows that $\rad(KM)^{n-1}\neq 0$, whence $KM$ has Loewy length $n$.
\end{proof}

The following theorem concerns lifting quiver presentations from nonmodular characteristic $p>0$ to characteristic $0$.  The astute reader will note that the proof only uses that $p$ is non-strongly modular and that $J/J^2$ has no $p$-torsion, which potentially is weaker than nonmodular.

\begin{Thm}\label{t:quiver.pres}
Let $M$ be a finite monoid and $p>0$ a prime which is nonmodular for $M$.   Let $(F,\mathcal O,k)$ be a splitting $p$-modular system for $M$. Let $e$ be a basic idempotent for $kM$.  Then there is an idempotent $\til e\in \mathcal OM$ lifting $e$ such that $\til e$ is a basic idempotent of $FM$.  Moreover, there is a quiver $Q$ and an ideal $I$ of $\mathcal OQ$ such that
$F\otimes_{\mathcal O} I$ is an admissible ideal of $FQ$ with $FQ/(F\otimes_{\mathcal O} I)\cong \til eFM\til e$  and $k\otimes_{\mathcal O} I$ is an admissible ideal of $kQ$ with $kQ/(k\otimes_{\mathcal O} I)\cong ekMe$.
\end{Thm}
\begin{proof}
Note that $FM$ and $kM$ are both split algebras.  Let $e$ be a basic idempotent of $kM$.  That means there is a decomposition $1=\eta_1+\cdots+\eta_s$ of the identity of $kM$ into orthogonal primitive idempotents such that $\eta_1,\ldots, \eta_t$ are pairwise nonconjugate, representing all the conjugacy classes of primitive idempotents of $kM$, and $e=\eta_1+\cdots+\eta_t$.    Then since $\mathcal O$ is complete, we can lift this to a decomposition $1=\til\eta_1+\cdots+\til\eta_s$ of the identity of $\mathcal OM$ into orthogonal primitive idempotents.  We will still have $\til\eta_1,\ldots, \til \eta_t$ are pairwise nonconjugate and represent all the conjugacy classes of primitive idempotents of $\mathcal OM$ since $\mathcal OM$ is semiperfect, $\mathcal OM/\rad(\mathcal OM)\cong kM/\rad(kM)$ and $\mathcal OM\til\eta_i$ is the projective cover of $S_i=kM\eta_i/\rad(KM)\eta_i$ as an $\mathcal OM$-module. By Theorem~\ref{t:char.nonmodular}, each $FM\til\eta_i \cong F\otimes_{\mathcal O}\mathcal OM\til \eta_i$ is a projective indecomposable $FM$-module with simple top $F\otimes_{\mathcal O}\til S_i$, where $\til S_i$ is an $\mathcal OM$-lattice with $k\otimes_{\mathcal O}\til S_i\cong S_i$.  In particular, $F\otimes_{\mathcal O}\til S_i\cong F\otimes_{\mathcal O}\til S_j$ if and only if $S_i\cong S_j$.   It follows that $1=\til\eta_1+\cdots+\til\eta_s$ is a decomposition of the identity of $FM$ into orthogonal primitive idempotents and that $\til e=\til\eta_1+\cdots+\til \eta_t$ is a basic idempotent of $FM$.  Also note that $\dim_F \til eFM\til e=\dim_k ekMe$, as both are obtained from tensoring over $\mathcal O$ with the $\mathcal O$-order $\til e\mathcal OM\til e$ in $\til eFM\til e$.

The exact sequence of $\mathbb ZM$-bimodules
\begin{equation}\label{eq:rad2.seq}
0\to J^2\to J\to J/J^2\to 0
\end{equation}
 leads to an exact sequence $0\to \mathcal O\otimes_{\mathbb Z}J^2\to \mathcal O\otimes_{\mathbb Z} J\to \mathcal O\otimes_{\mathbb Z}J/J^2\to 0$ since $\mathcal O$ has characteristic $0$ and hence is flat over $\mathbb Z$.
 Also, using that $\mathcal O$ is flat over $\mathbb Z$, we have that $\mathcal O\otimes_{\mathbb Z}J$ embeds isomorphically as an $\mathcal OM$-bimodule in $\mathcal OM$ as the $\mathcal O$-module $\til J$ spanned by $J$ over $\mathcal O$, which is an ideal.  Composing the injective maps $\mathcal O\otimes_{\mathbb Z}J^2\to \mathcal O\otimes_{\mathbb Z}J$ and $\mathcal O\otimes_{\mathbb Z}J\to \mathcal OM$, we see that $\mathcal O\otimes_{\mathbb Z}J^2$ is isomorphic as an $\mathcal OM$-bimodule to that $\mathcal O$-span of $J^2$, which is precisely $\til J^2$.   Thus $\til J/\til J^2\cong (\mathcal O\otimes_{\mathbb Z} J)/(\mathcal O\otimes_{\mathbb Z} J^2)\cong \mathcal O\otimes_{\mathbb Z}J/J^2$ as a bimodule.

 We claim that $\Tor_1^{\mathcal O}(k,\til J/\til J^2)=0$.  Indeed, the natural map $k\otimes_{\mathbb Z}J^2\to kM$ is injective by Proposition~\ref{p:loewy.preserved} and factors through the natural map $k\otimes_{\mathbb Z}J^2\to k\otimes_{\mathbb Z}J$.  This yields
exactness of $0\to k\otimes_{\mathbb Z}J^2\to k\otimes_{\mathbb Z} J\to k\otimes_{\mathbb Z} (J/J^2)\to 0$.  Therefore,  $0\to k\otimes_{\mathcal O}\til J^2\to k\otimes_{\mathcal O} \til J\to k\otimes_{\mathcal O} \til J/\til J^2\to 0$ is exact by transitivity of extension of scalars (identifying $\til J^i$ with $\mathcal O\otimes_{\mathbb Z}J^i$ for $i=1,2$ and $\til J/\til J^2$ with $\mathcal O\otimes_{\mathbb Z} J/J^2$), whence $\Tor_1^{\mathcal O}(k,\til J/\til J^2)=0$ as required.  Since $\til J/\til J^2$ is finitely generated over $\mathcal O$, and $\mathcal O$ is a discrete valuation ring with residue field $k$, we deduce that $\til J/\til J^2$ is a free $\mathcal O$-module.  Note that $\rad(kM)/\rad^2(kM)\cong k\otimes_{\mathcal O}\til J/\til J^2$ and $\rad(FM)/\rad^2(FM)\cong  F\otimes_{\mathcal O} \til J/\til J^2$ as bimodules by Proposition~\ref{p:modular.rad} and transitivity of extension of scalars.

Now $\til\eta_i [\til J/\til J^2]\til \eta_j$ is a direct summand in $\til J/\til J^2$ and hence is a free $\mathcal O$-module.  Thus we have, for each $i,j$, an exact sequence of free $\mathcal O$-modules $0\to \til\eta_i\til J^2\til \eta_j\to \til \eta_i\til J\til \eta_j \to \til\eta_i [\til J/\til J^2]\til \eta_j\to 0$ which splits over $\mathcal O$.  Moreover, these tensor with $F$ and $k$ to give compatible splittings of \[0\to \til\eta_i\rad^2(FM)\til \eta_j\to \til \eta_i\rad(FM)\til \eta_j \to \til\eta_i [\rad(FM)/\rad^2(FM)]\til \eta_j\to 0\] over $F$ and of \[0\to \eta_i\rad^2(kM)\eta_j\to \eta_i\rad(kM)\eta_j \to \eta_i [\rad(kM)/\rad^2(kM)]\eta_j\to 0\] over $k$.  Thus we can find a lift $\til B_{ij}$ of an $\mathcal O$-basis for $\til\eta_i [\til J/\til J^2]\til \eta_j$ to $\til \eta_i\til J\til \eta_j$ that projects to a subset $B_{ij}$  of $\eta_i[\rad(kM)]\eta_j$ mapping to a basis of the $k$-vector space $\eta_i[\rad(kM)/\rad^2(kM)]\eta_j$ and also projects to an $F$-basis of $\til\eta_i [\rad(FM)/\rad^2(FM)]\til \eta_j$.

Let $Q$ be the quiver with vertex set $\{1,\ldots, t\}$ and with $|\til B_{ij}|$ arrows from $j$ to $i$.
  Define a homomorphism $\psi\colon \mathcal OQ\to \til e\mathcal OM\til e$ by sending the empty path at vertex $i$ to $\til\eta_i$ and choosing a bijection between the arrows from $j$ to $i$ and $\til B_{ij}$.  Let $I$ be the kernel of this homomorphism.  Then the composition of $\psi$ with the projection to $ekMe$ induces a homomorphism $\psi'\colon kQ\to ekMe$ that is surjective by quiver theory  with kernel an admissible ideal $I'$ (by the above properties of $B_{ij}$).  It follows that $\psi(\mathcal OQ)+\mathfrak m\til e\mathcal OM\til e=\til e\mathcal OM\til e$, and so $\psi$ is surjective by Nakayama's lemma (as $\til e\mathcal OM\til e$ is finitely generated an $\mathcal O$-module).  Thus $\til e\mathcal OM\til e\cong \mathcal OQ/I$ and since $0\to I\to \mathcal OQ\to \til e\mathcal OM\til e\to 0$ is an exact sequence of free $\mathcal O$-modules, we also have that $0\to k\otimes_{\mathcal O} I\to kQ\to ekMe\to 0$ is exact, and so $k\otimes_{\mathcal O}I$ is the admissible ideal $I'$.  Also, $F\otimes_{\mathcal O}I$ is the kernel of the homomorphism $FQ\to \til eFM\til e$ obtained by extending the scalars, which is surjective with admissible kernel by quiver theory by the above properties of $\til B_{ij}$.  This completes the proof.
\end{proof}

\section{Characters and Brauer characters}
Brauer characters for monoids were first considered, independently, in~\cite{guralnick} and~\cite{PutchaBrauer}.  The approach of the former reference defines the Brauer character on the entire monoid, and is really what should be called the ``Brauer lift'' of a Brauer character. The approach of~\cite{PutchaBrauer} is equivalent to ours, but formulated differently.

 Let $\mu_r(R)$ denote the group of $r^{th}$-roots of unity in a commutative ring $R$.
Suppose that $(F,\mathcal O,k)$ is a $p$-modular system for $M$ such that $F$ contains a primitive $n^{th}$-root of unity where $n$ is the least common multiple of the periods of the elements of $M$.  So, in particular, the $p$-modular system is splitting by Theorem~\ref{t:Brauer+CMP}. Write $n=n'p^v$ with $\gcd(n',p)=1$.  Then  notice $\mu_{n'}(F)=\mu_{n'}(\mathcal O)$ and the reduction map $\rho\colon \mathcal O\to k$ induces an isomorphism $\mu_{n'}(\mathcal O)\to \mu_{n'}(k)$.  If $\omega\in \mu_{n'}(k)$, we shall write $\wh \omega$ for its preimage in $\mu_{n'}(\mathcal O)$ under $\rho$.

The following observation goes back (independently) to McAlister~\cite{McAlisterCharacter} and Rhodes and Zalcstein~\cite{RhodesZalc}, at least in characteristic $0$.

\begin{Prop}\label{p:RZMC}
Let $K$ be a field of characteristic $p$ (possibly $p=0$) and $M$ a finite monoid.  Suppose that $K$ contains a primitive $n^{th}$-root of unity where $n$ is the least common multiple of the periods of the elements of $M$.    Let $\rho\colon M\to M_r(K)$ be a representation.  If $m\in M$  is $p$-regular, then $\rho(m)$ is similar to a block diagonal matrix \[\begin{bmatrix} N &0\\ 0 & D\end{bmatrix}\] where $N$ is an upper triangular matrix with zeroes on the diagonal and $D$ is a diagonal matrix with diagonal entries in $\mu_{n'}(K)$ where we put $n'=n$ if $p=0$, and otherwise factor $n=n'p^v$ with $\gcd(n',p)=1$.  The matrix $D$ is unique up to reordering the diagonal.
\end{Prop}
\begin{proof}
Suppose that $m$ has index $i$ and period $s$.  Our assumption says that $p$ is not a divisor of $s$.  Then since $m^i=m^{i+s}$, it follows that the minimal polynomial $p(x)$ of $\rho(m)$ divides $x^i(x^s-1)$.  Since $s\mid n$ and $p\nmid s$, we conclude that $x^s-1$ has $s$ distinct roots in $K$, all of which belong to $\mu_{n'}(K)$.  It follows that the Jordan canonical form of $\rho(m)$ over $K$ exists and is of the desired form.  The uniqueness of $D$ follows from the uniqueness of the Jordan canonical form, as $D$ is the direct sum of the Jordan blocks corresponding to the nonzero eigenvalues of $\rho(m)$.
\end{proof}

Let us return to the situation of a $p$-modular system $(F,\mathcal O,k)$ for $M$ such that $F$ contains a primitive $n^{th}$-root of unity where $n$ is the least common multiple of the periods of the elements of $M$. Factor $n=n'p^v$ with $\gcd(n',p)=1$.    If $V$ is a finitely generated $kM$-module and $m\in M$ is $p$-regular, then Proposition~\ref{p:RZMC} shows that the portion of the Jordan canonical form of $m$ acting on $V$ corresponding to the nonzero eigenvalues has the form $\mathrm{diag}(\lambda_1,\ldots, \lambda_t)$ where $\lambda_1,\ldots, \lambda_t\in \mu_{n'}(k)$.  We define the \emph{Brauer character} $\theta_V$ of $V$ by $\theta_V(m) = \wh \lambda_1+\cdots+\wh \lambda_t\in \mathcal O$.  Thus the Braeur character can be viewed as a map from the set of $p$-regular elements of $M$ to $\mathcal O\subseteq F$.  If $V$ is simple, then $\theta_V$ is called an \emph{irreducible Brauer character}.  Let us say that a mapping from the $p$-regular elements of $M$ to $F$ is a \emph{$p$-regular class function} if it is  constant on $p$-regular generalized conjugacy classes.

The following basic properties of Brauer characters are proved almost exactly as in the group case~\cite{Webbbook,CurtisReinerI}.

\begin{Prop}\label{p:Brauer.chars}
Let $(F,\mathcal O,k)$ be a $p$-modular system for $M$ such that $F$ contains a primitive $n^{th}$-root of unity where $n$ is the least common multiple of the periods of elements of $M$.
\begin{enumerate}
\item Brauer characters of $kM$-modules are $p$-regular class functions.
\item If $V$ is a finitely generated $kM$-module, then $\theta_V(1)=\dim_k V$.
\item If $\chi_V$ is the $k$-valued character of a $kM$-module $V$, then $\theta_V(m) +\mathfrak m=\chi_V(m)$ for $m$ $p$-regular.
\item If $0\to U\to V\to W\to 0$ is an exact sequence of finitely generated $kM$-modules, then $\theta_V=\theta_U+\theta_W$.
\item If $U,V$ are finitely generated $kM$-modules, then, for $m$ $p$-regular, $\theta_{U\otimes_k V}(m) = \theta_U(m)\theta_V(m)$.
\end{enumerate}
\end{Prop}
\begin{proof}
As usual, write $n=n'p^v$ with $\gcd(n',p)=1$.
Suppose that $a,b\in M$ are $p$-regular with $a\sim b$.  Then we can find $x,x'\in M$ with $xx'x=x$, $x'xx'=x'$, $x'x=a^{\omega}$, $xx'=b^{\omega}$ and $xa^{\omega+1}x'=b^{\omega+1}$.  Let $V$ be a finitely generated $KM$-module. First note that $\theta_V(a)=\theta_V(a^{\omega+1})$.  To see this, let $i$ be the index and $s$ the period of $a$.  Then $s\mid n'$ because $a$ is $p$-regular.  Thus $a^{\omega}= a^{in'}$, and so $a^{\omega+1}=a^{in'+1}$.  Let $\rho\colon M\to M_r(k)$ be a matrix representation afforded by $V$ with respect to some basis.  By Proposition~\ref{p:RZMC}, $\rho(a)$ is similar to a block diagonal matrix with diagonal blocks $N,D$ with $N$ nilpotent and $D$ diagonal with entries in $\mu_{n'}(k)$.  But then $D^{in'+1}=D$, whereas $N^{in'+1}$ is nilpotent.  Thus $\rho(a^{\omega+1})=\rho(a)^{in'+1}$ has the same nonzero eigenvalues as $\rho(a)$ with multiplicities.  Thus $\theta_V(a)=\theta_V(a^{\omega+1})$ and $\theta_V(b)=\theta_V(b^{\omega+1})$ (the latter by the same reasoning).  Note that $b^{\omega+1}=x(a^{\omega+1}x')$ and $(a^{\omega+1}x')x= a^{\omega+1}a^{\omega}=a^{\omega+1}$.  Therefore, $\theta_V(a^{\omega+1}) =\theta_V(b^{\omega+1})$ by the well-known fact that if $A,B\in M_r(k)$, then $AB$ and $BA$ have the same characteristic polynomials (and hence the same eigenvalues with multiplicity). The easiest proof of this fact is to set \[C=\begin{bmatrix} \lambda I & A\\ B & I\end{bmatrix}, D=\begin{bmatrix}I & 0\\ -B & \lambda I \end{bmatrix},\] and note that $\lambda^r\det(\lambda I-AB) = \det (CD)=\det (DC)=\lambda^r\det(\lambda I-BA)$.

The second and third items are immediate from the definition.  For the remaining items, it is worth noting the following interpretation of Proposition~\ref{p:RZMC}.  Let $m\in M$ be $p$-regular with index $i$ and period $s$ and put $e=m^{\omega}$.   Note that $s\mid n'$ and $p\nmid s$.   Then $e$ is a central idempotent of $k\langle m\rangle\cong k[x]/(x^{i+s}-x^i)$, $k\langle m\rangle e\cong k[x]/(x^s-1)\cong k^s$ and $k\langle m\rangle(1-e)\cong k[x]/(x^i)$.  In particular, if $V$ is a finitely generated $kM$-module, then $V=(1-e)V\oplus eV$ as a $k\langle m\rangle$-module, and $m$ acts nilpotently on $(1-e)V$ and diagonally on $eV$ with eigenvalues in $\mu_{n'}(k)$.  Now if $0\to U\to V\to W\to 0$ is an exact sequence of finitely generated $KM$-modules, then $0\to eU\to eV\to eW\to 0$ is an exact sequence of $k\langle m\rangle e$-modules.  Since $k\langle m\rangle e\cong k^s$, we deduce that $eV\cong eU\oplus eW$ as $k\langle m\rangle e$-modules, and so the multiset of eigenvalues of $m$ on $eV$ is the union of the multisets of its eigenvalues on $eU$ and $eW$.  Since $m$ has only zero eigenvalues on $(1-e)U$, $(1-e)V$ and $(1-e)W$, the fourth item follows.

 The fifth item is argued as follows.  Retaining the above notation for $m$ $p$-regular, $e(U\otimes_k V) \cong eU\otimes_k eV$ as $k\langle m\rangle e$-modules.  But if $eU$ has a basis of eigenvectors $u_1,\ldots, u_r$ of $m$ with eigenvalues $\lambda_1,\ldots, \lambda_r$ and  $eV$ has a basis of eigenvectors $v_1,\ldots, v_t$ of $m$ with eigenvalues $\nu_1,\ldots, \nu_t$, then $eU\otimes_k eV$ has basis of eigenvectors $u_i\otimes v_j$ for $m$, with $1\leq i\leq r$, $1\leq j\leq t$, and corresponding eigenvalues $\lambda_i\nu_j$.  As, $\wh{\lambda_i\nu_j}=\wh{\lambda_i}\wh{\nu_j}$, it follows that $\theta_{U\otimes_k V}(m)=\theta_U(m)\theta_V(m)$, as required.
\end{proof}

It follows from Proposition~\ref{p:Brauer.chars}(4)--(5) that a Brauer character is a nonnegative integer combination of irreducible Brauer characters as $[V]\mapsto \theta_V$ is a well-defined ring homomorphism from $G_0(kM)$ to the ring of $p$-regular class functions.

Our notion of Brauer character is essentially the same as that of Putcha~\cite{PutchaBrauer}, except that he only defines Brauer characters on $p$-regular elements of index $0$.  Since $m\sim m^{\omega+1}$ and $m^{\omega+1}$ has index $0$,  there is no loss of information in defining the Brauer character only on such elements.  Putcha also uses complex roots of unity rather than $p$-modular systems.

Brauer characters behave well under restriction to maximal subgroups.
\begin{Prop}\label{p:restrict.brauer}
Let $V$ be a finitely generated $kM$-module and $e\in E(M)$. Then $\theta_V|_{G_e} = \theta_{eV}$.
\end{Prop}
\begin{proof}
Note that $V=eV\oplus (1-e)V$ as a $k$-vector space.  Elements of $G_e$ annihilate $(1-e)V$ and act by invertible operators on $eV$.  It follows that the nonzero eigenvalues of $g\in G_e$ on $V$ and $eV$ are the same (with multiplicities), and so $\theta_V(g)=\theta_{eV}(g)$.
\end{proof}

If $\theta_V$ is the Brauer character of a $kM$-module $V$, we define the \emph{Brauer lift} of $\theta_V$ to be the mapping $\til\theta_V\colon M\to \mathcal O$ defined by $\til\theta_V(m) = \theta_V(m^{\omega+1}_{p'})$.  This is the definition used for a Brauer character in~\cite{guralnick}, except that they work with complex roots of unity instead of $p$-modular systems.

\begin{Thm}\label{t:brauer.lift}
Let $V$ be a $kM$-module with Brauer character $\theta_V$.  Then $\til\theta_V$ is a virtual character of $M$ (over $F$) and $\til\theta_V(m) +\mathfrak m =\chi_V(m)$ for all $m\in M$, where $\chi_V$ is the $k$-valued character of $V$.
\end{Thm}
\begin{proof}
Note that $\til\theta_V$ takes values in $\mathcal O$.  It follows from the results of~\cite{ourcharacter} that $\chi_V(m)=\chi_V(m^{\omega+1}_{p'})$, and so $\til\theta_V(m) +\mathfrak m =\chi_V(m)$ by Proposition~\ref{p:Brauer.chars}.  The results of~\cite{ourcharacter} also show that if $a,b$ are generalized conjugates, then $a^{\omega+1}_{p'}$ and $b^{\omega+1}_{p'}$ are generalized conjugates.  It then follows that $\til\theta_V$ is a class function.  By~\cite[Corollary~7.16]{repbook} (which is stated for $\mathbb C$, but holds for any splitting field of characteristic $0$) or~\cite[Theorem~3.2]{ourcharacter}, it then suffices to show that $\til\theta_V|_{G_e}$ is a virtual character of $G_e$ for each maximal subgroup $G_e$.  If $g\in G_e$, then $g^{\omega+1}=g^{\omega}g=eg=g$.  Therefore, $\til\theta_V(g) = \theta_V(g_{p'}) = \theta_{eV}(g_{p'})$ by Proposition~\ref{p:restrict.brauer}.  Thus $\til\theta_V|_{G_e}$ is the standard Brauer lift of the Brauer character $\theta_{eV}$ of $G_e$.  But Brauer proved the standard Brauer lift is a virtual character (cf.~\cite[Theorem~18.12]{CurtisReinerI}).  This completes the proof.
\end{proof}

As a corollary, we show that the irreducible Brauer characters are linearly independent over $F$.  This was essentially done in~\cite{guralnick} for the Brauer lifts of Brauer characters.

\begin{Thm}\label{t:indep.brauer.chars}
Let $(F,\mathcal O,k)$ be a $p$-modular system for $M$ such that $F$ contains a primitive $n^{th}$-root of unity where $n$ is the least common multiple of the periods of elements of $M$. Then the irreducible Brauer characters of $M$ form a basis for the set of $p$-regular class functions.  Hence $F\otimes_{\mathbb Z} G_0(kM)$ is isomorphic to the ring of $p$-regular class functions via the map induced by $[V]\mapsto \theta_V$.  In particular, the Brauer character of a $kM$-module determines its composition factors.
\end{Thm}
\begin{proof}
Let $\pi\in \mathcal O$ generate the maximal ideal.
Let $T_1,\ldots, T_r$ represent the isomorphism classes of simple $kM$-modules.
First note that the irreducible $k$-valued characters $\chi_{T_1},\ldots,\chi_{T_r}$ are linearly independent over $k$.  This can be deduced easily since $k$ is a splitting field for $kM$, although it is true also for arbitrary fields (cf.~\cite{ourcharacter,guralnick}).   Suppose that $c_1\theta_{T_1}+\cdots+c_r\theta_{T_r}=0$ with $c_1,\ldots, c_r\in F$. By clearing denominators we may assume that $c_1,\ldots, c_r\in \mathcal O$.   By construction  $c_1\til\theta_{T_1}+\cdots+c_r\til\theta_{T_r}=0$.  Theorem~\ref{t:brauer.lift} then implies $c_1\chi_{T_1}+\cdots+c_r\chi_{T_r}=0$.  By $k$-linear independence of $\chi_{T_1},\ldots, \chi_{T_r}$, we deduce that $c_1,\ldots, c_r\in (\pi)$.   Assume inductively that $c_1,\ldots, c_s\in (\pi^s)$ with $s\geq 1$.  Then $d_1\til\theta_{T_1}+\cdots+d_r\til\theta_{T_r}=0$ with $d_i=c_i/\pi^s$.  The above argument shows that $d_1,\ldots, d_r\in (\pi)$, whence $c_1,\ldots, c_r\in (\pi^{s+1})$.  Since $\bigcap_{s\geq 1}(\pi^s)=\{0\}$, we deduce that $c_1=c_2=\cdots=c_r=0$, establishing the linear independence of $\theta_{T_1},\ldots, \theta_{T_r}$.

Since $r$ is the number of $p$-regular generalized conjugacy classes (cf.~\cite{ourcharacter}), it follows that the Brauer characters form a basis for the space of $p$-regular class functions.  Proposition~\ref{p:Brauer.chars} shows that $[V]\mapsto \theta_V$ is a well-defined ring homomorphism from $G_0(kM)$ to the ring of $p$-regular class functions.   The linear independence of the Brauer characters over $F$ shows that it induces an isomorphism of $F$-algebras after extending the scalars.
\end{proof}

Note that the final statement of the theorem was proved  by Putcha~\cite{PutchaBrauer} using the analogous result of Brauer for groups.

If $M$ is a monoid, let us denote by $\mathrm{Cl}_p(M)$ the ring of $F$-valued $p$-regular class functions on $M$. Of course, we are assuming here that the $p$-modular system is fixed.

\begin{Prop}\label{p:reduced.to.groups}
Let $e_1,\ldots, e_t$ be a set of representatives of the regular $\J$-classes of $M$.  Then the mapping $\Res\colon \mathrm{Cl}_p(M)\to \prod_{i=1}^r \mathrm{Cl}_p(G_{e_i})$ given by $\Res(f)_i = f|_{G_{e_i}}$ is an isomorphism of $F$-algebras.
\end{Prop}
\begin{proof}
This follows from the fact that each $p$-regular generalized conjugacy class of $M$ intersects exactly one $G_{e_i}$ and it intersects it in a $p$-regular conjugacy class.
\end{proof}

We remark that the restriction of an irreducible Brauer character of $M$ to a maximal subgroup is a Brauer character, but it need not be irreducible.  However, from Clifford-Munn-Ponizovskii theory, it does follow that if we use as a basis for $\mathrm{Cl}_p(M)$ the irreducible Brauer characters and  as a basis for $\prod_{i=1}^r \mathrm{Cl}_p(G_{e_i})$ the irreducible Brauer characters running over all the groups, then for a suitable ordering of the Brauer characters the matrix of $\Res$ is unipotent upper triangular.  The proof is the same as for the case of complex characters~\cite{repbook}.

One can also define the Brauer character table of a monoid $M$ to have rows indexed by irreducible Brauer characters and columns indexed by $p$-regular generalized conjugacy classes.  The Brauer character table is invertible by linear independence of the Brauer characters.  For suitable orderings of the irreducible Brauer characters and $p$-regular generalized conjugacy classes, the Brauer character table will be block upper triangular with the diagonal blocks the Brauer character tables of the maximal subgroups of $M$ (one per regular $\mathscr J$-class).  The proof is exactly the same as the analogous result for ordinary character tables, cf.~\cite{McAlisterCharacter,RhodesZalc} or~\cite[Chapter~7]{repbook}.

If $\chi$ is an irreducible character of $M$ over $F$, then its restriction to the set of $p$-regular elements is a $p$-regular class function and hence a linear combination of irreducible Brauer characters.  As is the case for groups, this linear combination is given by via the decomposition matrix.

\begin{Prop}\label{p:decomp.as.brauer}
Let $(F,\mathcal O,k)$ be a $p$-modular system for $M$ such that $F$ contains a primitive $n^{th}$-root of unity where $n$ is the least common multiple of the periods of elements of $M$.  Then we have a commutative diagram
\[\begin{tikzcd}
G_0(FM)\ar{r}{[V]\mapsto \chi_V}\ar{d}[swap]{d} & \mathrm{Cl}(M)\ar{d}{r}\\
G_0(kM)\ar{r}{[V]\mapsto \theta_V} & \mathrm{Cl}_p(M)
\end{tikzcd}\]
where $r$ is the map obtained by restriction to the $p$-regular elements of $M$.
\end{Prop}
\begin{proof}
Let $V$ be an $FM$-module and let $W$ be full $\mathcal OM$-lattice in $V$.  Let $m\in M$ be $p$-regular.  Let $\psi\colon M\to M_{\ell}(\mathcal O)$ be a matrix representation afforded by $W$ with respect to an $\mathcal O$-basis for $W$ (which is necessarily an $F$-basis for $V$).   The characteristic polynomial $p(x)$ of $\psi(m)$ belongs to $\mathcal O[x]$.   It follows from Proposition~\ref{p:RZMC} that the minimal polynomial $q(x)$ of $\psi(m)$ over $F$ is of the form $x^a\cdot \prod_{i=1}^t(x-\xi_i)$ where the $\xi_i$ are distinct $(n')^{th}$-roots of unity, where $n=n'p^v$ with $\gcd(n',p)=1$.  Therefore, $p(x)=x^j\cdot \prod_{i=1}^t(x-\xi_i)^{n_i}$ for some $j\geq a$ and $n_i\geq 1$.  Let $\rho\colon \mathcal O\to k$ be the projection. Then the characteristic polynomial of $m$ on $W/\mathfrak mW$ is $x^j\cdot \prod_{i=1}^t(x-\rho(\xi_i))^{n_i}$.  It follows that $\theta_{W/\mathfrak mW}(m) = \sum_{i=1}^t n_i\xi_i=\chi_V(m)$ by definition of the Brauer character associated to $W/\mathfrak mW$.  As $d([V]) = [W/\mathfrak mW]$, this completes the proof.
\end{proof}

Proposition~\ref{p:decomp.as.brauer}, in light of Theorem~\ref{t:brauer.lift}, provides another proof that the decomposition map is surjective (Theorem~\ref{t:decomp.surjective}).

We can now reinterpret Question~\ref{conjecture} in terms of character theory.   Let $(F,\mathcal O,k)$ be a $p$-modular system for $M$ such that $F$ contains a primitive $n^{th}$-root of unity where $n$ is the least common multiple of the periods of elements of $M$. The Cartan matrix $C(FM)$ (which coincides with $C(\mathbb CM)$) is nonsingular if and only if each finitely generated projective $FM$-module is determined up to isomorphism by its ordinary character.  On the other hand, the Cartan matrix $C(kM)$ is nonsingular if and only if each finitely generated projective $kM$-module is determined up to isomorphism by its Brauer character.  Let us now reformulate our question.

\begin{Question}\label{conj:characterversion}
Let $p>0$ be a prime and $(F,\mathcal O,k)$ a $p$-modular system for $M$ such that $F$ contains a primitive $n^{th}$-root of unity where $n$ is the least common multiple of the periods of elements of $M$.  Suppose that $C(FM)$ is nonsingular and that $P$ is a finitely generated projective $\mathcal OM$-module.  Under what conditions on $M$ must the ordinary character of $F\otimes_{\mathcal O}P$ be determined by its restriction to the set of $p$-regular elements of $M$?
\end{Question}

Let us formally prove the relationship between these two versions of the question.

\begin{Prop}\label{p:reform}
Let $p>0$ be a prime and $(F,\mathcal O,k)$ a $p$-modular system for $M$ such that $F$ contains a primitive $n^{th}$-root of unity where $n$ is the least common multiple of the periods of elements of $M$.  Suppose that $C(FM)$ is nonsingular and that, for every finitely generated projective $\mathcal OM$-module $P$, the ordinary character of $F\otimes_{\mathcal O}P$ be determined by its restriction to the set of $p$-regular elements of $M$.  Then $C(kM)$ is nonsingular.
\end{Prop}
\begin{proof}
 Theorem~\ref{t:decomp.surjective} shows that $e\colon K_0(kM)\to K_0(FM)$ is injective and our assumption says that the Cartan map $C(FM)\colon K_0(FM)\to G_0(FM)$ is injective.  Thus we need to verify that the decomposition map is injective on the image of $C(FM)e$ by Theorem~\ref{t:brauer.square}.  Let $P$ be a finitely generated projective $\mathcal OM$-module.  Then $C(FM)e([P/\mathfrak mP]) = [F\otimes_{\mathcal O} P]$ (where the right hand side is viewed as an element of $G_0(FM)$).  So we want to show that the class in $G_0(FM)$ of $F\otimes_{\mathcal O} P$ is determined up to isomorphism by $d([F\otimes_{\mathcal O} P])$.  But by Proposition~\ref{p:decomp.as.brauer}, this is the same as saying that the ordinary character of $F\otimes_{\mathcal O}P$ is determined by its restriction to the $p$-regular elements of $M$.
\end{proof}


Brauer showed that the ordinary character of a finitely generated projective $\mathcal OG$-module vanishes on every element that is not $p$-regular for a group $G$~\cite[Theorem~18.26]{CurtisReinerI}.  This, unfortunately, is not the case for monoids.

\begin{Example}
 Let $M$ be the multiplicative monoid of $\mathbb F_3$, which we think of as $\mathbb Z/3\mathbb Z$, and let $p=2$.  The $2$-regular elements of $M$ are $\ov 1,\ov 0$ (which have period 1), whereas $\ov 2$ has period $2$ (putting $\ov n = n+3\mathbb Z$).  Note that $(\mathbb Q_2,\mathbb Z_2,\mathbb F_2)$ is a $2$-modular system with $F$ containing a primitive square root of unity.  Then $\mathbb Z_2M$, itself, is a finitely generated projective module, and the character of $\mathbb Q_2M$ is given by $\chi(\ov 1)=3$, $\chi(\ov 2)= 1$ and $\chi(\ov 0)=1$.  Thus $\chi$ does not vanish on the element $\ov 2$ of period $2$.  The problem here is that $\mathbb Z_2M$ does not restrict to a projective module over $\mathbb Z_2\langle \ov 2\rangle$, as would happen if $M$ were a group.  Nonetheless, the ordinary character of any finitely generated projective $\mathbb Z_2M$-module is determined by its restriction to the $2$-regular elements of $M$.  Indeed, it is easy to see that the primitive idempotents of $\mathbb Z_2M$ are $\ov 0$ and $\ov 1-\ov 0$.  One has that $\mathbb Q_2M\ov 0=\mathbb Q_2\ov 0$ is the trivial module and $\mathbb Q_2M(\ov 1-\ov 0)$ is $\mathbb Q_2\mathbb F_3^\times$, viewed as a $\mathbb Q_2M$-module by having $\ov 0$ act as zero, and the elements of $\mathbb F_3^\times$ act as usual.  Hence if $P$ is a finitely generated projective indecomposable $\mathbb Z_2M$-module, then $\mathbb Q_2\otimes_{\mathbb Z_2}P\cong \mathbb Q_2^a\oplus \mathbb (\mathbb Q_2\mathbb F_3^\times)^b$ for some $a\geq 0$, $b\geq 0$.  If $\chi$ is the character of this module, one has $a=\chi(\ov 0)$ and $b=\frac{1}{2}(\chi(\ov 1)-\chi(\ov 0))$.  So $P$ is still determined by the value of $\chi$ on the $2$-regular elements $\ov 0,\ov 1$.
 \end{Example}

\section{Positive results on Question~\ref{conjecture}}
In this section we  present some positive results concerning Question~\ref{conjecture}.

\subsection{Regular monoids}
Recall that if $A$ is a finite dimensional algebra, we put $\cd (A)=\det C(A)$.  Our goal is to compute the Cartan determinant of a regular
monoid in terms of the Cartan determinants of its maximal subgroups.  This will in particular answer Question~\ref{conjecture} for such monoids.

The following  is a special case of~\cite[Proposition~2.1]{quasistrat}.

\begin{Prop}\label{p:control.glob.cd}
Let $A$ be a finite dimensional $K$ algebra and $e\in A$ an idempotent such that  $AeA$ is projective as a left $A$-module.
Then $\cd (A)=\cd (eAe)\cdot \cd (A/AeA)$ where $\cd (0)=1$ by convention.
\end{Prop}

\begin{Thm}\label{t:main.cd}
Let $M$ be a regular monoid and $K$ a field. Let $e_1,\ldots, e_t$ be idempotent representatives of the $\mathscr J$-classes of $M$.  Then \[\cd (KM)=\cd (KG_{e_1})\cd (KG_{e_2})\cdots \cd (KG_{e_t}).\]  Consequently, $\cd (KM)=1$ if the characteristic of $K$ does not divide the order of any maximal subgroup of $M$, and, otherwise, $\cd (KM)$ is a proper $p$-power where $p>0$ is the characteristic of $K$.  In particular, $C(KM)$ is nonsingular for all fields $K$.
\end{Thm}
\begin{proof}
The second statement follows from the first, since group algebras are semisimple in nonmodular characteristic and Brauer proved that, in modular characteristic, the Cartan determinant of a group algebra is a proper $p$-power~\cite[Theorem~21.24]{CurtisReinerI}.  To prove the first statement, let \[\emptyset=I_0\subsetneq I_1\subsetneq\cdots \subsetneq I_t=M\] be a principal series.  We may assume that the idempotents are chosen so that $I_k\setminus I_{k-1} = J_{e_k}$.  Let $A_k = KM/KI_k$ for $i=0,\ldots, t-1$.  Note that $A_{t-1}\cong KG_{e_t}$ and $A_0=KM$.  For $0\leq k\leq t-1$, we have $KI_{k+1}/KI_k=A_ke_{k+1}A_k$, $e_{k+1}A_ke_{k+1}\cong KG_{e_{k+1}}$, $A_k/(KI_{k+1}/KI_k)\cong A_{k+1}$ and, by~\cite[Lemma~5.18]{repbook},  $KI_{k+1}/KI_k\cong (A_ke_{k+1})^{\ell_{k+1}}$ as left $A_k$-modules, where $\ell_{k+1}$ is the number of $\mathscr L$-classes in $J_{k+1}$, and hence is a projective left $A_k$-module. Therefore, $\cd(A_k)=\cd (KG_{e_{k+1}})\cd (A_{k+1})$ by Proposition~\ref{p:control.glob.cd}.  Downward  induction yields $\cd (A_k) = \cd (KG_{e_t})\cdots \cd(KG_{e_{k+1}})$ for $0\leq k\leq t-1$.  Taking $k=0$, establishes $\cd (KM)=\cd (KG_{e_1})\cdots \cd (KG_{e_t})$.
\end{proof}

\subsection{Monoids with aperiodic stabilizers}

If $p>0$ is a prime, we say that a finite monoid $M$ is a \emph{$p'$-monoid} if no element of $M$ has period divisible by $p$, that is, $p$ does not divide the order of any maximal subgroup of $M$.  A finite monoid $M$ is called \emph{aperiodic} if all its elements have period $1$, that is, all its maximal subgroups are trivial.

If $M$ is a monoid and $a\in M$, then the left stabilizer of $a$ is the submonoid $\mathrm{Stab}_L(a)=\{m\in M\mid ma=a\}$.   The right stabilizer $\mathrm{Stab}_R(a)$ is defined dually.   For example, if $M$ is a group, then each left and right stabilizer is trivial.

The following group theoretic result is well known, cf.~\cite[Corollary~11.3.6]{Webbbook}.

\begin{Prop}\label{p:relatively.proj}
Let $R$ be a commutative ring with unit, $G$ a finite group and $H$ a subgroup.  Suppose that $V$ is an $RG$-module whose restriction to $RH$ is projective and $[G:H]$ is a unit in $R$.  Then $V$ is a projective $RG$-module.
\end{Prop}

We can now generalize Brauer's result~\cite[Theorem~18.26(ii)]{CurtisReinerI} to a class of monoids.

\begin{Prop}\label{p:p'stablizer}
Let $M$ be a monoid such that the left stabilizer of any element of $M$ is a $p'$-prime monoid.  Let $(F,\mathcal O,k)$ be a $p$-modular system for $M$ such that $F$ contains a primitive $n^{th}$-root of unity where $n$ is the least common multiple of the periods of elements of $M$.  Then if $P$ is a finitely generated projective $\mathcal OM$-module and $\chi$ is the ordinary character of $F\otimes_{\mathcal O} P$, then $\chi(m)=0$ for every element $m\in M$ that is not $p$-regular.
\end{Prop}
\begin{proof}
Suppose that $m\in M$ is not $p$-regular.  Then $m,m^{\omega+1}$ are generalized conjugates, $m^{\omega+1}$ is not $p$-regular and $\chi(m)=\chi(m^{\omega+1})$.  Thus we may assume without loss of generality that $m=m^{\omega+1}$.  Let $e=m^{\omega}$ and let $G=\langle m\rangle$, which is a cyclic group generated by $m$ with identity $e$ and order divisible by $p$.  We claim that $eP$ is a finitely generated projective $\mathcal OG$-module.

Since each finitely generated projective $\mathcal OM$-module is a direct sum of projective indecomposable ones, and each projective indecomposable $\mathcal OM$-module is a direct summand in $\mathcal OM$, it suffices to show that $e\mathcal OM$ is a finitely generated projective $\mathcal OG$-module.  Write as usual,  $m=m_pm_{p'}=m_{p'}m_p$ with $m_p\in G$ an element of $p$-power order and $m_{p'}\in G$ $p$-regular.  Let $H=\langle m_p\rangle$.  Then $H$ is a $p$-group, and so our assumption on $M$ guarantees that $H$ acts freely on $eM$.  Hence $e\mathcal OM$ is a free $\mathcal OH$-module.   Since $[G:H]$ is not divisible by $p$, it is a unit in $\mathcal O$ (recall that $\mathfrak m\cap \mathbb Z =p\mathbb Z$, as $k$ has characteristic $p$).  Thus $e\mathcal OM$ is a projective $\mathcal OG$-module by Proposition~\ref{p:relatively.proj}, and clearly finitely generated as $eM$ is finite.

  Now $\chi|_G$ is the character of $e(F\otimes_{\mathcal O} P)=F\otimes_{\mathcal O}eP$ as an $FG$-module by the analogue of Proposition~\ref{p:restrict.brauer} for ordinary characters~\cite{repbook} (cf.~\cite[Proposition~7.14]{repbook}).  But then Brauer's result~\cite[Theorem~18.26(ii)]{CurtisReinerI} applies to show that $\chi(m)=0$, as $m$ is a not a $p$-regular element of $G$ and $eP$ is $\mathcal OG$-projective.
\end{proof}

Note that under the hypotheses of Proposition~\ref{p:p'stablizer} and the assumption that $C(FM)$ is nonsingular, the ordinary characters of finitely generated projective $\mathcal OM$-modules form a basis for the space of class functions on $M$ vanishing on the non-$p$-regular generalized conjugacy classes as they are linearly independent (since $e\colon K_0(kM)\to K_0(FM)$ is injective and $C(FM)$ is nonsingular) and there are as many of them as there are $p$-regular conjugacy classes.

\begin{Cor}\label{c:conjecture.good}
Let $p>0$ be a prime.  Suppose that $M$ is a monoid such that the left (or dually right) stabilizer of every element is a $p'$-monoid.  If $C(\mathbb CM)$ is nonsingular, then $C(kM)$ is nonsingular for any field $k$ of characteristic $p$.
\end{Cor}
\begin{proof}
By Proposition~\ref{p:split.field.enough}, it is enough to handle the case that $k$ is a splitting field.  We may then restrict to the case that the left stabilizer of each element is a $p'$-monoid, as $C(A^{op}) = C(A)^T$ for any finite dimensional algebra $A$ over a splitting field.  Since the Cartan matrix over any splitting field of characteristic $0$ coincides with $C(\mathbb CM)$ (up to conjugation by a permutation matrix),
 it is enough to show that if  $(F,\mathcal O,k)$ is a $p$-modular system for $M$ such that $F$ contains a primitive $n^{th}$-root of unity, where $n$ is the least common multiple of the periods of elements of $M$, and $C(FM)$ is nonsingular, then $C(kM)$ is nonsingular.   By Proposition~\ref{p:reform}, it is enough to prove  that if $P$ is a finitely generated projective $\mathcal OM$-module, then the ordinary character of $F\otimes_{\mathcal O} P$ is determined by its values on $p$-regular elements of $M$.  But this is immediate from Proposition~\ref{p:p'stablizer}.
\end{proof}

If $M$ is a $p'$-monoid, then obviously its stabilizers are also $p'$-monoids, and so Corollary~\ref{c:conjecture.good} can be viewed as an improvement on Corollary~\ref{c:good.char}.

Notice that $M$ has aperiodic left stabilizers if and only if the left stabilizers are $p'$-monoids for every prime $p$.  Hence we have the following consequence of Corollary~\ref{c:conjecture.good}.

\begin{Cor}\label{c:aperiodic.stab}
Let $M$ be a finite monoid such that the left (or dually right) stabilizer of every element is aperiodic.  If $C(\mathbb CM)$ is nonsingular, then $C(KM)$ is nonsingular for every field $K$.
\end{Cor}

Every finite monoid has a unique minimal (nonempty) ideal, which moreover is a regular $\mathscr J$-class, cf.~\cite[Chapter~1]{repbook}.  In particular, the left stabilizer of each element has a minimal ideal.  If $M$ is a regular monoid and $m\in M$, then $m=mnm$ for some $n\in M$.  Then $e=mn$ is an idempotent in $\mathrm{Stab}_L(m)$, and also $\mathrm{Stab}_L(m)=\mathrm{Stab}_L(e)$.  In particular, $e$ is a right zero of $\mathrm{Stab}_L(m)$ (and hence belongs to its minimal ideal), and the maximal subgroup at $e$ in $\mathrm{Stab}_L(m)$ is trivial.  If $\mathrm{Stab}_L(m)$ is a $p'$-monoid, then obviously the maximal subgroup of its minimal ideal is a $p'$-group.  This leads to the following conjecture.

\begin{Conjecture}\label{conj:new.conj}
Let $p>0$ be a prime.  Let $M$ be a finite monoid and suppose that the maximal subgroup of the minimal ideal of the left (or dually right) stabilizer of each element of $M$ is a $p'$-group. Then nonsingularity of $C(\mathbb CM)$ implies nonsingularity of $C(kM)$ for any field $k$ of characteristic $p$.
\end{Conjecture}

Nonsingularity of the Cartan matrix of the algebra of a regular monoid (cf.~Theorem~\ref{t:main.cd}) and Corollary~\ref{c:conjecture.good} are both implied by Conjecture~\ref{conj:new.conj}, and hence can be viewed as evidence towards that conjecture.

A proposed answer to Question~\ref{conjecture} is then the following.

\begin{Conjecture}\label{conj:new.conj.two}
Let $M$ be a finite monoid and suppose that the maximal subgroup of the minimal ideal of the left (or dually right) stabilizer of each element of $M$ is trivial. Then nonsingularity of $C(\mathbb CM)$ implies nonsingularity of $C(KM)$ for any field $K$.
\end{Conjecture}

\def\malce{\mathbin{\hbox{$\bigcirc$\rlap{\kern-7.75pt\raise0,50pt\hbox{${\tt
  m}$}}}}}\def\cprime{$'$} \def\cprime{$'$} \def\cprime{$'$} \def\cprime{$'$}
  \def\cprime{$'$} \def\cprime{$'$} \def\cprime{$'$} \def\cprime{$'$}
  \def\cprime{$'$} \def\cprime{$'$}

\end{document}